\documentclass[final,leqno]{siamltex}
\usepackage{amsmath}
\usepackage{amssymb}
\usepackage{graphicx}
\usepackage{color}
\usepackage{amsfonts}
\usepackage{amscd}
\usepackage{mathrsfs}
\usepackage{bm}
\usepackage{hyperref}
\usepackage{enumerate}
\usepackage{eurosym}
\usepackage{algorithmic, algorithm}
\usepackage{comment}

\usepackage{tikz}
\usetikzlibrary{decorations.pathreplacing, decorations.pathmorphing, decorations.shapes}
\usetikzlibrary{arrows,calc}

\newcommand{\beq}{\begin{equation}}
\newcommand{\eeq}{\end{equation}}
\newcommand{\D}{\displaystyle}

\newcommand{\bx}{\bm{x}}
\newcommand{\be}{\bm{e}}

\newcommand{\by}{\bm{y}}

\newcommand{\bs}{\bm{s}}

\newcommand{\al}{\alpha}
\newcommand{\ep}{\epsilon}

\newcommand{\bfb}{{\bm{b}}}
\newcommand{\bff}{{\bm{f}}}

\newcommand{\bfu}{{\bm{ u}}}

\newcommand{\mcL}{{\mathcal{L}}}

\newcommand{\mcA}{{\mathcal{A}}}
\newcommand{\mcD}{{\mathcal{D}}}
\newcommand{\mcG}{{\mathcal{G}}}

\newcommand{\cS}{{\mathcal{S}}}
\newcommand{\cV}{{\mathcal{V}}}

\newcommand{\del}{\delta}

\newcommand{\mathd}{\mathrm{d}}

\newcommand{\vomega}{\vec{\omega}}
\newcommand{\vod}{\vomega_\delta}

\newcommand{\omd}{\underline{\omega}_\delta}

\newcommand{\hod}{\hat{\omega}_\delta}
\newcommand{\R}{\mathbb{R}}
\newcommand{\Z}{\mathbb{Z}}

\newcommand{\Od}{{\Omega}}
\newcommand{\Om}{{\Omega}}

\newtheorem{thm}{Theorem}
\newtheorem{prop}[thm]{Proposition}
\newtheorem{coro}[thm]{Corollary}
\newtheorem{lem}[thm]{Lemma}
\newtheorem{remark}[thm]{\textit{Remark}}

\newtheorem{assu}[thm]{Assumption}

\begin{document}

\title{Mathematics of Smoothed Particle Hydrodynamics: a Study via Nonlocal Stokes Equations
\thanks{To appear in Foundation of Computational Mathematics, 2019.\quad This work is supported in part by the U.S.~NSF grants
DMS-1719699 and DMS-1819233,  AFOSR MURI center for material failure prediction through peridynamics, and ARO MURI Grant W911NF-15-1-0562.
}}

\author{ Qiang Du\thanks{Department of Applied Physics and Applied Mathematics, 
Columbia University, New York, NY 10027,    \href{mailto:qd2125@columbia.edu}{qd2125@columbia.edu}}
 \and Xiaochuan Tian\thanks{Department of Mathematics, University of Texas, Austin,
TX 78712,  \href{mailto:xtian@math,utexas.edu}{xtian@math,utexas.edu}
}}

\markboth{NONLOCAL STOKES EQUATION}{QIANG DU AND XIAOCHUAN TIAN}

%\vspace{-1.5in}
\date{}%Received: 7 August  2017 /  Accepted: 1 July 2019}

\maketitle

\begin{abstract}
%\vspace{-1.in}
Smoothed Particle Hydrodynamics (SPH) is a popular numerical technique developed for simulating
complex fluid flows. Among its key ingredients is the use of nonlocal integral  relaxations to local differentiations.
Mathematical analysis of the corresponding nonlocal models on the continuum level can 
provide further theoretical understanding of SPH.
We present, in this part of a series of works on the mathematics of SPH,
a nonlocal relaxation to the conventional linear steady state Stokes system 
for incompressible viscous flows. The nonlocal continuum model is
characterized by a smoothing length $\delta$ which measures the range of nonlocal interactions.
It serves as a bridge between the discrete approximation schemes that involve a nonlocal
integral relaxation and the local continuum models.
We show that for a class of carefully chosen nonlocal operators, the resulting
nonlocal Stokes equation is well-posed
and recovers the original Stokes equation in the local limit when
$\delta$ approaches zero. {For some other commonly used smooth kernels, there are risks in
getting ill-posed continuum models that could lead to computational difficulties in 
practice. }  This leads us to discuss the implications of our finding on
the design of numerical methods.\\[.2cm]
AMS (MOS) SUBJECT CLASSIFICATION: 45P05,\  45A05,\  35A23,\ 
 46E35
\end{abstract}


%%%%%%%%%%%%%%%%%%%%%%%%%%%%%%%%%%%%%%%%%%%%%%%%%%%%%%%%%%%%%%%
%%%%%%%%%%%%%%%%%%%%%%%%%%%%%%%%%%%%%%%%%%%%%%%%%%%%%%%%%%%%%%%
%%%%%%%%%%%%%%%%%%%%%%%%%%%%%%%%%%%%%%%%%%%%%%%%%%%%%%%%%%%%%%%
\section{Introduction}

There has been much recent interest in nonlocal continuum models. In solid mechanics,
the theory of peridynamics \cite{sill00}
was proposed as a possible  alternative to conventional models of elasticity and fracture mechanics.
It has also been shown to be an integral relaxation to the
conventional models when the latter are valid such as the case of linear elasticity.
Mathematical and numerical analysis of peridynamics have provided a solid theoretical foundation to {nonlocal mechanical models and their numerical approximations} \cite{du12,MeDu15a,TiDu14}.
 In this work, we are interested in 
extending such mathematical studies to problems in fluid mechanics. Indeed, nonlocal integral relaxations are naturally linked to numerical schemes developed
for simulating fluid flows such as the smoothed particle hydrodynamics (SPH) \cite{GiMo77,liu2010,lucy77,mono95}
vortex  methods \cite{BM85,CK00}  and others \cite{Cher17,AB00,ELC02,Koum05,TE04}. 
While developed originally for astrophysical applications, SPH has become
a very popular computational technique for simulating complex flows, including both compressible and incompressible
flows, meanwhile, it has also encountered issues like the loss of stability and lack of resolution.
The mathematical analysis of SPH remains limited today except \cite{mv00,ld19}. 
We present, as part of a series of investigations on the mathematics of SPH, 
a nonlocal relaxation to the conventional linear steady state Stokes system 
for incompressible viscous flows, with the aim of providing further insight into the theoretical foundation of methods like SPH.
The  proposed nonlocal models serve as bridges linking SPH with the local differential equation models and
 allow us to delineate
effects resulted from the different aspects of the approximation process. This adds new angle to the subject that has not
been systematically explored in the literature before.
The nonlocal equation studied here is also different from but closely connected to other nonlocal and fractional fluid models
such as those considered for geostrophic 
flows \cite{con06,con08} and hyper-dissipative models
\cite{katz02,tao} as well as hydrodynamic limit of kinetic models \cite{tadmor14}. 

One of the main contributions of  this work is to formulate a well-posed nonlocal analog of the linear Stokes system (in two and three space
dimensions).  In general, 
a spatially nonlocal model may depend on the laws of nonlocal interactions specified in the bulk spatial region and necessary modifications near
the boundary involving possible boundary conditions or nonlocal constraints \cite{du12}.
As a first step, we will be focusing on the bulk nonlocal interactions
 in this work, namely, finding suitable nonlocal relaxations of the first and second order differential operators  used in the local Stokes equation
such that the resulting nonlocal Stokes equation remains well-posed  and retains a consistent local limit.
To that end, we only consider periodic boundary conditions to avoid the discussion near physical boundary.  
Such a simplification allows us to use Fourier analysis to carry out much needed technical derivations. 
Though the analysis would work in any space dimension, our attention is on two and three dimensional spaces
for physical relevance and this also complements a similar discussion in  the 
one dimensional case presented in \cite{dt17} for the correspondence model of peridynamic materials.
The nonlocal analogs of differential operators adopted in this paper are basic elements of nonlocal vector calculus
introduced in \cite{du13} and have been successfully applied to nonlocal modeling and analysis \cite{du12,MeDu15}. 
Our main finding in this work is that %while the nonlocal interaction kernel in the diffusion operator can be very general, 
the choices of the nonlocal gradient and divergence operators are more subtle issues that need more careful treatment. 
More specifically, we reveal that nonlocal interaction kernels associated with nonlocal gradient and divergence operators should have suitably strengthened interactions as particles (materials points) get close together, {which in this work is phrased as {\it strong nearby interactions} for easy reference}. Many popular choices used in
practical implementation of SPH, however, do not yield such type of interactions,  {thus leading to possibly ill-posed problems on the continuum level in the smoothing step.  Although suitable numerical discretization may add
regularization effect to help 
alleviating the impact of any intrinsic ill-posedness, it is expected that these numerical regularization effect would be highly dependent on the resolution level (like particle distribution). Such speculations would require further theoretical investigation. Nevertheless, 
by revealing potential flaws in the key smoothing step for developing a robust and practical SPH methodology, our analysis provides strong links between the popular SPH discretization and a new mathematical foundation of nonlocal operators.}

{To further highlight the bridging role of the nonlocal models in the analysis of SPH type methods and to show that 
the ill-posedness of nonlocal models can be avoided with carefully designed nonlocal operators, 
we present examples of well-posed nonlocal Stokes models involving nonlocal gradient and divergence operators with suitably strengthened nearby interactions. For these nonlocal models, we actually show that the spaces of nonlocal and local divergence-free vector fields  coincide, a nice property-preserving feature of the nonlocal relaxation.
We also get, as a byproduct, that using the nonlocal kernels under consideration, the nonlocal Laplacian in terms of the composition of nonlocal
gradient and divergence operators is a well-defined and invertible operator which may be used in either the nonlocal Stokes system or a scalar Poisson equation. The latter is often used to do the pressure correction for maintaining incompressibility
in practical SPH implementation.  With well-posed nonlocal Stokes
equation, to make further connections with the second step of discretizing the nonlocal operators in SPH, 
 we deduce}
the convergence and the asymptotic compatibility of the Fourier spectral approximations. The latter is
natural due to the special periodic setting. {Moreover,  the mathematical findings in this special case demonstrate the possibility  that, once well-posed nonlocal continuum models can be constructed from the smoothing step, it is possible to develop robust discretization
that can maintain the convergence in different parameter regimes (for example, either for a fixed smoothing length $\delta>0$ 
or for $\delta \to 0$, as the numerical resolution improves}. It also
serves as a hint for future analysis of collocation and mesh-free methods like SPH
that are originally designed for more complex geometric settings. {Indeed, one may design other remedies
such as using nonlocal gradient operators with biased nonlocal interactions (similar in spirit to upwind differences 
\cite{dlee19}) or a formulation involving artificial compressibility
\cite{zds19},  The latter two subjects are explored in separate works.}
In short, our study here
further illustrates that
the mathematics developed for nonlocal continuum models may offer foundation for a better 
understanding of related numerical discretizations. 

The rest of the paper is organized as follows. We begin by formulating the nonlocal
 linear Stokes system in section \ref{sec:nonlocal-stokes}, in which the  notation and assumptions are introduced. 
 We then establish the well-posedness of the resulting nonlocal
model in section \ref{sec:wellpose} and the connection to the usual local Stokes system in section \ref{sec:vanish}. 
In section \ref{sec:num}, we discuss the Fourier spectral methods and related convergence issues. We conclude in section \ref{sec:discu} 
with a summary and a discussion on ongoing and future research. 

%%%%%%%%%%%%%%%%%%%%%%%%%%%%%%%%%%%%%%%%%%%%%%%%%%%%%%%%%%%%%%%
%%%%%%%%%%%%%%%%%%%%%%%%%%%%%%%%%%%%%%%%%%%%%%%%%%%%%%%%%%%%%%%
%%%%%%%%%%%%%%%%%%%%%%%%%%%%%%%%%%%%%%%%%%%%%%%%%%%%%%%%%%%%%%%

\section{Nonlocal Stokes system}
\label{sec:nonlocal-stokes}

Let us first recall the conventional, local Stokes equation.
Let $\bfu$ be the velocity field, $p$ the pressure, $\bff$ the body force,  $\nu$ the given viscosity
coefficient and a bounded domain $\Omega\subset \mathbb{R}^n$ with a smooth boundary $\partial\Omega$,  
 the conventional, local Stokes 
equation of interests here refers to the system 
\begin{align}
\label{eq:stokes}
\left\{\begin{array}{rcl}
\D  - \nu  \Delta \bfu + \nabla p &= & \bff\,,\quad\text{in}\;\; \Omega\\[.2cm]
\D -\nabla \cdot \bfu & = &0\,,\quad \text{in}\;\; \Omega\,.
\end{array}\right.
\end{align}
%For simplicity, we consider \eqref{eq:stokes} with the usual no-slip boundary  condition
%\begin{align}
%  \label{eq:bc}
%\bfu=0,\quad\quad \text{on}\;\; \p\Omega \;.
%\end{align}
The linear Stokes equation  \eqref{eq:stokes} serves as a simplification of the stationary
 and time-dependent Navier-Stokes equations which are
among  the most widely studied mathematical models for fluid flows.
In recent years,  there have been much interests in various generalizations and
relaxations of the linear Stokes and nonlinear Navier-Stokes systems, in particular,
those involving nonlocal operators. We refer to earlier mentioned examples like gestrophic 
flows \cite{con06,con08}, hyper-dissipative models
\cite{katz02,tao} and
 hydrodynamic limit of kinetic models \cite{tadmor14}.
 Of particular interests to us is the
nonlocal relaxation used in the SPH for the simulation of complex flows.
While developed for astrophysical applications, SPH for incompressible viscous flows has also been
a subject of ongoing study \cite{antu12,eller07,hu07,hu09,Lee08,nair15,poz02,SL03}.
 Despite broad applications and much 
progress in the algorithmic development effort, there has not much rigorous examination on the 
underlying relaxed nonlocal continuum models which could serve as bridges between 
local continuum models and their numerical
discretizations.

We begin by focusing on a linear nonlocal system defined on the periodic cell given 
by $\Od=(-\pi,\pi)^d\subset\R^d$ in dimension $d=2$ or $3$.
Given a parameter small parameter $\delta>0$ representing the nonlocal interaction length,
we define the nonlocal Stokes equation as: for a given periodic function $\bff$ on $\Od$, find
periodic functions $\bfu_\delta$ and $p_\delta$ such that
\begin{align}
  \label{eq:stokes-nonlocal}
\left\{\begin{array}{l}
\D  - \nu \mcL_\delta \bfu_\delta (\bx) + \mcG_\delta p_\delta (\bx) = \bff(\bx), \qquad \bx\in \Om, \\[.2cm]
\D - \mcD_\delta \bfu_\delta(\bx) =0,  \qquad \bx\in \Om, 
\end{array}\right.
\end{align}
with normalization conditions on  $\bfu_\delta$ and $p_\delta$ to eliminate constant shifts, and on $\bff$ to assure compatibility
\begin{align}
\label{eq:constraint}
\int_{\Om} \bfu_\delta(\bx) d \bx=0,\qquad
\D \int_{\Om}p_\delta(\bx)\mathd\bx=0,
\qquad  \int_{\Om} \bff(\bx) d \bx=0
\,.
\end{align}

The nonlocal operators used there are the nonlocal diffusion operator $\mcL_\delta$, nonlocal gradient operator
$\mcG_\delta$ and nonlocal divergence operator $\mcD_\delta$ given respectively by
\begin{align}
  \label{eq:op-nonlocal}
  \D
  \mcL_\delta \bfu (\bx) = 
\int_{\R^d} {\omd(|\bx-\by|)}(\bfu(\by)-\bfu(\bx))\mathd \by,\\
\D  \mcG_\delta p (\bx) =   \int_{\R^d} \vod(\by-\bx) (p(\by)-p(\bx))\mathd \by\\
\D  \mcD_\delta \bfu (\bx) = \int_{\R^d}  \vod(\by-\bx)\cdot (\bfu(\by)+\bfu(\bx))\mathd \by\,.
\end{align}
The above operators are determined by a nonlocal scalar-valued kernel $\omd$ and a
vector-valued kernel  $\vod$. In this work, we take a special
form $\vod(\bx)=\hod(|\bx|)\bx/|\bx|$, and for the moment, we assume that 
$\omd$ and $\hod$ both are nonnegative, radial symmetric and 
with a  compact support in the $\delta$ neighborhood $B(\bf{0},\delta)$ of the origin.
Here, $\delta$ is a parameter that characterizes the range of nonlocal interaction. It is
called a {\em nonlocal horizon parameter} in peridynamics, following \cite{sill00}, but in SPH, it is more commonly
called the {\em smoothing length}.
While more specific assumptions
on the nonlocal kernels are given later, we note that these operators have been studied extensively
in recent years, see for instance, \cite{du13,MeDu14e,MeDu14r,MeDu15} for more discussions and
generalizations. The development of these operators, and in particular, their vector and tensor forms, has
been motivated by the mathematical analysis of peridynamics {and other nonlocal  integral equation
models}, though the connection to SPH has also
been alluded to previously \cite{dlt15}. In more general mathematical context, 
these nonlocal operators serve as  the nonlocal analog of the classical diffusion, gradient and divergence operators,
upon properly choosing the nonlocal kernels, and 
they form part of the building blocks of the nonlocal vector calculus, together with relevant integral identities. Indeed,
by a nonlocal integration by parts formula (see \cite[Theorem 2.7]{MeDu15}), we have $\mcG_\delta$ and $\mcD_\delta$ 
to be adjoint operators of each other, in the sense that, for functions in 
the suitable spaces (and periodic in our context),
\begin{align}
  \label{eq:adjoint}
\int_{\Od} \bfu (\bx) \cdot \mcG_\delta p (\bx) d\bx = - \int_{\Od}   \mcD_\delta \bfu (\bx)  p (\bx) d\bx \,,
\end{align}
 just like the conventional gradient and divergence operators.  {In addition, the expression $\bfu(\by)+\bfu(\bx)$ in
 the definition of $  \mcD_\delta \bfu$ 
 in \eqref{eq:op-nonlocal} can be used interchangeably with $\bfu(\by)- \bfu(\bx)$, although the plus sign is
 preferred in the more general function class over a bounded domain $\Od$ to have \eqref{eq:adjoint} satisfied.}
 
 We note that 
 integral relaxations to differential operators 
  have also been discussed in many works related to the SPH methods as mentioned earlier,
   except that they are more often given by approximate
  quadrature forms
 in disguise.  In the SPH community, 
 $\delta$ is called the smoothing length, thus we refer $\delta$ as either horizon parameter or
 smoothing length interchangeably in this work.
 
 Nonlocal operators such as $\mcL_\delta$  can also be seen as
continuum forms of popular discrete and graph Laplacians (see e.g. \cite{Hein05,Shi17,diffusion-map}).
The study on $\mcG_\delta$ and $\mcD_\delta$, as explained in this work, bears even greater significance for the nonlocal
Stokes equation.  Such a study is also related to the so-called correspondence
theory of peridynamics, see a recent study in \cite{dt17}.

 Now we specify conditions on the kernels $\omega_\del$ and $\hat\omega_\del$ used in the definition of the nonlocal diffusion operator $\mcL_\del$, and
 the nonlocal gradient and
 divergence operators $\mcG_\delta$ and $\mcD_\delta$.

\begin{assu}
\label{assumption}    The nonnegative and radial symmetric kernels $\omd=\omd(|\bx|)$ and $\hod=\hod(|\bx|)$ are assumed to satisfy
the following assumptions.

\begin{enumerate}

\item  The kernels $\omd=\omd(|\bx|)$ and $\hod=\hod(|\bx|)$ have compact support in the sphere $B(\bf{0},\delta)$ and satisfy the normalization conditions:
\begin{align}
\label{omdmnt}
 \frac{1}{2} \int_{\mathbb{R}^n} 
\omd (|\bx|)  |\bx|^2 \mathd \bx = d\,.
 \end{align}
and
 \begin{align}
\label{hodmnt}
  \int_{\mathbb{R}^n} 
\hod (|\bx|)  |\bx| \mathd \bx = d.
 \end{align}

\item The kernels $\omd$ and $\hod$ are rescaled from kernels $\omega$ and $\hat\omega$ that have compact support in the unit sphere:
\begin{align} \label{kernel_rescale}
\left\{\begin{array}{c}
 \D \omd(|\bx|)= \frac{ 1 }{\delta^{d+2}}\omega\left(\frac{|\bx|}{\delta}\right), \\[.2cm]
  \D
  \hod(\bx) =   \frac{ 1 }{\delta^{d+1}}\hat\omega\left(\frac{|\bx|}{\delta}\right).
 \end{array}\right.
 \end{align}
\end{enumerate}

\end{assu}

The above conditions are often made in earlier studies on the nonlocal operators and are also
default assumptions throughout this paper.  {In the literature on SPH and vortex-blob  methods, these
moment conditions have also been used as well, see, e.g., \cite{CK00}.}
 It turns out that,  as shown
in the next section, additional conditions on the kernels, particularly on
$\hod$,
are needed in order to obtain a well-posed nonlocal Stokes system.

%%%%%%%%%%%%%%%%%%%%%%%%%%%%%%%%%%%%%%%%%%%%%%%%%%%%%%%%%%%%%%%
%%%%%%%%%%%%%%%%%%%%%%%%%%%%%%%%%%%%%%%%%%%%%%%%%%%%%%%%%%%%%%%
%%%%%%%%%%%%%%%%%%%%%%%%%%%%%%%%%%%%%%%%%%%%%%%%%%%%%%%%%%%%%%%

%%%%%%%%%%%%%%%%%%%%%%%%%%%%%%%%%%%%%%%%%%%%%%%%%%%%%%%%%%%%%%%
%%%%%%%%%%%%%%%%%%%%%%%%%%%%%%%%%%%%%%%%%%%%%%%%%%%%%%%%%%%%%%%
%%%%%%%%%%%%%%%%%%%%%%%%%%%%%%%%%%%%%%%%%%%%%%%%%%%%%%%%%%%%%%%

\section{Well-posedness of the nonlocal Stokes system with periodic condition}
\label{sec:wellpose} 

In this section, we examine the well-posedness of 
 the nonlocal Stokes system given by \eqref{eq:stokes-nonlocal}. To begin with, let us define an operator associated with
 the vector system by
\begin{align}
\mcA_\delta=\left(\begin{array}{cc} - \nu \mcL_\delta & \mcG_\delta\\
-\mcD_\delta &0\end{array}\right)\,.
\end{align}

Under periodic conditions and the constraints \eqref{eq:constraint}, we write $\bfu$ and $p$ in terms of their Fourier series, namely,
\[
\bfu(\bx) = \frac{1}{(2\pi)^d}\sum_{\bm\xi\in \Z^d, \bm\xi\neq0} \widehat{\bfu}(\bm\xi) e^{i\bm\xi \cdot \bx}  \;\text{and }  \; 
p(\bx)= \frac{1}{(2\pi)^d}\sum_{\bm\xi\in \Z^d,\bm\xi\neq0} \widehat{p}(\bm\xi) e^{i\bm\xi \cdot \bx} \,,
\]
where
\[
 \widehat{\bfu}(\bm\xi)  =\int_\Od \bfu(\bx)  e^{-i \bm\xi\cdot \bx} d\bx  \;\text{and }  \;  \widehat{p}(\bm\xi)= \int_\Od p(\bx)e^{-i \bm\xi\cdot \bx} d\bx\,.
\]

The following lemma provides the associated Fourier symbols of the nonlocal operators.

\begin{lem} \label{lem:fourier}
The Fourier transform of operators $\mcL_\del$, $\mcG_\delta$ and $\mcD_\delta$ are given by
\begin{align}
&\widehat{\mcL_\del \bfu} (\bm\xi)= -\lambda_\del(\bm\xi)  \widehat{\bfu}(\bm\xi) \\
&\widehat{\mcG_\del p} (\bm\xi)=i \bfb_\del(\bm\xi) \widehat{p}(\bm\xi)\\
&\widehat{\mcD_\del \bfu}(\bm\xi)=i (\bfb_\del(\bm\xi))^T \widehat{\bfu}(\bm\xi) \,,
\end{align}
where $ \lambda_\del(\bm\xi) $ and $\bfb_\del(\bm\xi) $ are given by
\begin{align}
 \lambda_\del(\bm\xi)&= \int_{|\bs|\leq \del} \omd(|\bs|)(1-\cos(\bm\xi\cdot \bs)) d\bs \label{coef_lamb} \\
\bfb_\del(\bm\xi)  &= \int_{|\bs|\leq \del} \hod(|\bs|)\frac{\bs}{|\bs|}\sin(\bm\xi\cdot \bs) d\bs \,. \label{coef_b}
\end{align}
\end{lem}

\begin{proof}
The results follow immediately from the definitions of $\mcL_\del$, $\mcG_\delta$ and $\mcD_\delta$. 
\end{proof}

%Lemma \ref{lem:fourier} expresses the Fourier coefficients of the nonlocal operators in the form of integrals involving kernel functions.
For convenience,   in the following lemma, we further express $\lambda_\del(\bm\xi)$ and $\bfb_\del(\bm\xi)$ in
\eqref{coef_lamb} and \eqref{coef_b} using polar coordinates.

\begin{lem}
The Fourier symbols
$ \lambda_\del(\bm\xi)$  and $\bfb_\del(\bm\xi) $ in \eqref{coef_lamb} and \eqref{coef_b} can be equivalently expressed as
\begin{equation} \label{coef_lamb_polar}
\lambda_\del(\bm\xi)=
\left\{ 
\begin{aligned}
&4  \int_0^{\pi/2} \int_0^\del r \omd(r)\big(1- \cos(r \cos(\phi)|\bm\xi|) \big) dr d\phi  \quad \text{for } d=2,\\
&4\pi \int_0^{\pi/2} \sin(\phi) \int_0^\del r^2 \omd(r)\big(1- \cos(r \cos(\phi)|\bm\xi|) \big) dr d\phi  \quad \text{for } d=3,
\end{aligned}
\right.
\end{equation}
and
\begin{equation}
\bfb_\del(\bm\xi)=  b_\del(|\bm\xi|)\frac{\bm\xi}{|\bm\xi|} \label{coef_b_2}\,,
\end{equation}
where the scalar coefficient $b_\del(|\bm\xi|)$ is given by
\begin{equation}\label{coef_b_scalar}
b_\del(|\bm\xi|)=
\left\{ 
\begin{aligned}
&4  \int_0^{\pi/2} \cos(\phi) \int_0^\del r \hod(r)\sin(r \cos(\phi)|\bm\xi| ) dr d\phi  \quad \text{for } d=2,\\
&4\pi \int_0^{\pi/2}\cos(\phi) \sin(\phi) \int_0^\del r^2 \hod(r) \sin(r \cos(\phi)|\bm\xi| ) dr d\phi \quad \text{for } d=3.
\end{aligned}
\right.\,.
\end{equation}
\end{lem}

\begin{proof}
Let us show \eqref{coef_b_2} with $d=3$. The case with $d=2$ is similar and omitted. 
First we observe that for any orthogonal matrix $R$, we have
\[
\bfb_\del(\bm\xi)= R^T\bfb_\del(R\bm\xi)\,.
\]
Now denote $\be=(0,0,1)$. Let $R$ be the rotation matrix which rotates $\bm\xi$ to be aligned with $\be$, namely
\[
R\bm\xi = |\bm\xi| \be\,.
\]
Then $R\bm\xi\cdot\bs=|\bm\xi| s_3$ and 
\[
\bfb_\del(\bm\xi)= \int_{|\bs|\leq \del} \hod(|\bs|)\frac{R^T\bs}{|\bs|}\sin(|\bm\xi|s_3) d\bs \,.
\]
So each component of $\bfb_\del(\bm\xi)$ is given by
\[
\begin{split}
[\bfb_\del(\bm\xi)]_i &= \int_{|\bs|\leq \del}\frac{ \hod(|\bs|)}{|\bs|} \sum_{j=1}^3 R_{ji} s_j\sin(|\bm\xi|s_3) d\bs \\
&= \int_{|\bs|\leq \del}\frac{ \hod(|\bs|)}{|\bs|}R_{3i} s_3 \sin(|\bm\xi|s_3) d\bs \quad\text{for } i=1,2,3\,.
\end{split}
\]

The next task at hand  is to find $\{R_{3i}, i=1,2,3\}$. 
We know that $R$ is the matrix obtained by rotating $\bm\xi$ by an angle of $\arccos(\be\cdot \frac{\bm\xi}{|\bm\xi|})=\arccos({\xi_3}/{|\bm\xi|})$ around the axis in the direction of 
$$\frac{\bm\xi \times\be}{|\bm\xi \times\be|}=\frac{1}{\sqrt{\xi_1^2+\xi_2^2}}(\xi_2,-\xi_1,0).$$
 Such a rotation matrix can be explicitly
constructed. In particular, $R_{3i}$ is given by
\[
R_{3i}= \frac{\xi_i}{|\bm\xi|}\,, \quad \text{ for } i=1, 2,3\,. 
\]

Combing the above arguments, we obtain
\[
\begin{split}
\bfb_\del(\bm\xi)&=\frac{\bm\xi}{|\bm\xi|}\int_{|\bs|\leq \del} \hod(|\bs|) \frac{s_3}{|\bs|}  \sin(|\bm\xi|s_3) d\bs \\
&=\frac{\bm\xi}{|\bm\xi|} \left( 2\pi \int_0^{\pi}\cos(\phi) \sin(\phi) \int_0^\del r^2 \hod(r) \sin(r \cos(\phi)|\bm\xi| ) dr d\phi\right) \\
&= \frac{\bm\xi}{|\bm\xi|}  b_\del(\bm\xi)\,,
\end{split}
\]
where $b_\del(\bm\xi)$ is given by \eqref{coef_b_scalar}. 

Now \eqref{coef_lamb_polar} can be obtained similarly by noticing that $\lambda_\del(\bm\xi)=\lambda_\del(R\bm\xi)$  for any orthogonal matrix $R$. 
%\qed
\end{proof}

Via Fourier analysis, we now get a $(d+1)\times (d+1)$ matrix system:  
\[
A_\del(\bm\xi) \left( \begin{array}{c} \widehat{\bfu}(\bm\xi) \\  \widehat{p}(\bm\xi) \end{array} \right) =   \left( \begin{array}{c} \widehat{\bff}(\bm\xi) \\ 0 \end{array} \right)
\]
where 
\[ A_\del(\bm\xi)  =  \left( \begin{array}{cc}  \lambda_\del(\bm\xi) I_d& i \bfb_\del(\bm\xi) \\ -i (\bfb_\del(\bm\xi))^T  &0 \end{array} \right) \,.
\]
For each fixed $\bm\xi$, in order for the matrix $ A_\del(\bm\xi) $ to be invertible, we need $\lambda_\del(\bm\xi)$ and $\bfb_\del(\bm\xi)$ to be nonzero. 
Under such assumptions, the inverse of $ A_\del(\bm\xi) $ is given by
\begin{equation} \label{matrix_inverse}
(A_\del(\bm\xi))^{-1}=\left( \begin{array}{cc} \frac{1}{ \lambda_\del(\bm\xi)} \D \left( I_d-\frac{\bfb_\del(\bm\xi)\otimes \bfb_\del(\bm\xi)}{|\bfb_\del(\bm\xi)|^2}\right)& \D i \frac{\bfb_\del(\bm\xi)}{|\bfb_\del(\bm\xi)|^2} \\ 
\D -i \frac{(\bfb_\del(\bm\xi))^T}{|\bfb_\del(\bm\xi)|^2}   & \D \frac{-\lambda_\del(\bm\xi)}{|\bfb_\del(\bm\xi)|^2} \end{array} \right)\,.
\end{equation}

Now we can easily  observe from the expression in equation \eqref{coef_lamb_polar} that 
 $\lambda_\del(\bm\xi) $ is positive for any $\bm\xi$.
It is however more delicate to check the non-degeneracy of  $\bfb_\del(\bm\xi)$ since the latter
 is an integral of a product of a positive function with  a sign-changing oscillatory function. 
The following lemma (also utilized in \cite{dt17}) gives a simple observation on the sine Fourier coefficient that is useful to our discussion 
on the positivity of  $\bfb_\del(\bm\xi)$.

\begin{lem} \label{lem:positivity}
Given a measurable, non-negative and non-increasing function $g=g(x)$ with $xg(x)$ integrable, we have
\beq
\label{fou1}
\int_0^{2\pi} g(x) \sin(x) dx \geq 0\,
\eeq
with the equality holds only for $g$ being a constant function. Consequently, for any 
 $h>0$ and $a>0$, we have
\beq
\int_0^h g(x) \sin(ax) dx \geq 0\,,
\eeq
with the equality holds only for $g$ being a constant function (and with value zero if the product $ha$ is not an integer multiple of $2\pi$).
\end{lem}
\begin{proof} The inequality \eqref{fou1} follows immediately from the observation that
\[\int_0^{2\pi} g(x) \sin(x) dx  = \int_0^{\pi} [g(x)-g(x+\pi)] \sin(x) dx \geq 0.
\]
By the non-increasing property, we see that the  equality holds only for $g$ being a constant function. 
The more general case follows by applying a change of variable and taking a zero extension of $g$ outside $(0,h)$ to
cover complete periods of the scaled sine function.  %\qed
\end{proof}

A simple consequence of Lemma \ref{lem:positivity} and \eqref{coef_b_scalar}  is that $\bfb_\del(\bm\xi)$ is positive for 
any fixed $\bm\xi$ if $r^{d-1} \hat\omega(r)$ is a non-increasing function. 
This simple fact gives us a hint that in order for the nonlocal Stokes system to be well-posed, 
one should expect the nonlocal interaction in the gradient and divergence operators be suitably strengthened 
for physical points (particles) in closer proximity.  Now to offer a precise energy estimate, in the rest of this section, we 
assume that the kernel  $\hat\omega_\del(r)$ satisfies the following additional conditions.

\begin{assu} \label{assu:frackernel}
The kernel $\hat\omega_\del(r)$ is the rescaling of $\hat\omega(r)$ given by \eqref{kernel_rescale} with $\hat\omega(r)$ satisfying the following conditions. 
\begin{enumerate}
\item $r^{d-1}\hat\omega(r)$ is nonincreasing for $r\in (0,1)$;
\item $\hat\omega(r)$  is of fractional type in at least a small neighborhood of origin, namely, 
 there exists some $\ep>0$ such that for $s\in (0,\ep)$ we have
\begin{equation} \label{frackernel}
\hat\omega(r)= \frac{c}{r^{d+\beta}} \,,
\end{equation}
for some constant $c>0$ and $\beta\in (-1,1)$.
\end{enumerate} 
\end{assu}

\begin{remark}
The condition that   $r^{d-1}\hat\omega(r)$ is nonincreasing gives us a sufficient condition for $\bfb_\del(\bm\xi)$ to stay positive for finite $|\bm\xi|$. It is not a necessary condition, 
but it is shown later  by some examples that $\bfb_\del(\bm\xi)$ can be zero for finite $|\bm\xi|$ if the nonincreasing condition is violated, thus the nonlocal Stokes system is ill-posed in such cases. {Consequently, SPH schemes based on the nonlocal gradients with bounded and smooth kernels obtained from the smoothing step may contain intrinsic instabilities
that are difficult to eliminate on the discretization level, especially when the particle distribution becomes uneven or 
highly disordered.}

As an illustration, in Figure~\ref{fig:b} we plot the values of $b_\del(|\bm\xi|)$ against $|\bm\xi|$ for the kernel $\hat\omega(r)=\frac{1}{r^{d+\beta}}$ with $\beta< -1$,
 using the expression in \eqref{coef_b_scalar} with $d=2$ and $d=3$ respectively.
We observe from the plots the tendency for $b_\del(|\bm\xi|)$ to stay positive with the kernel $\hat\omega(r)$ being more singular at zero.  
 On the other hand, clear numerical evidence from the plots show that $b_\del(|\bm\xi|)$ may become zero 
 at some finite frequencies for $\beta<-1.5 $ in the two dimensional case  
 and for  $\beta< -2$ in the three dimensional case. 
\begin{figure}[h!]
  \centering
 \includegraphics[width=0.45\textwidth]{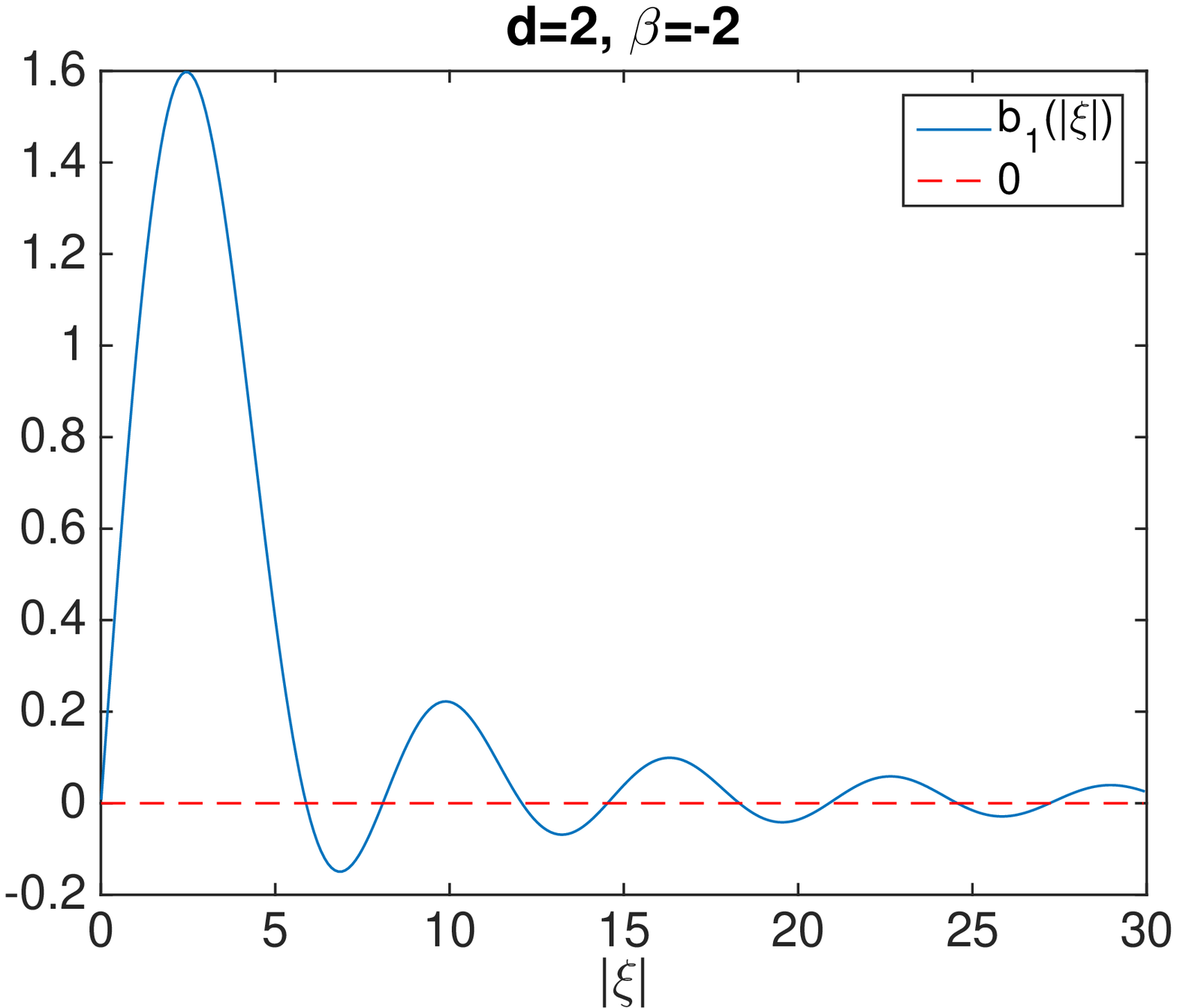}\quad   \includegraphics[width=0.45\textwidth]{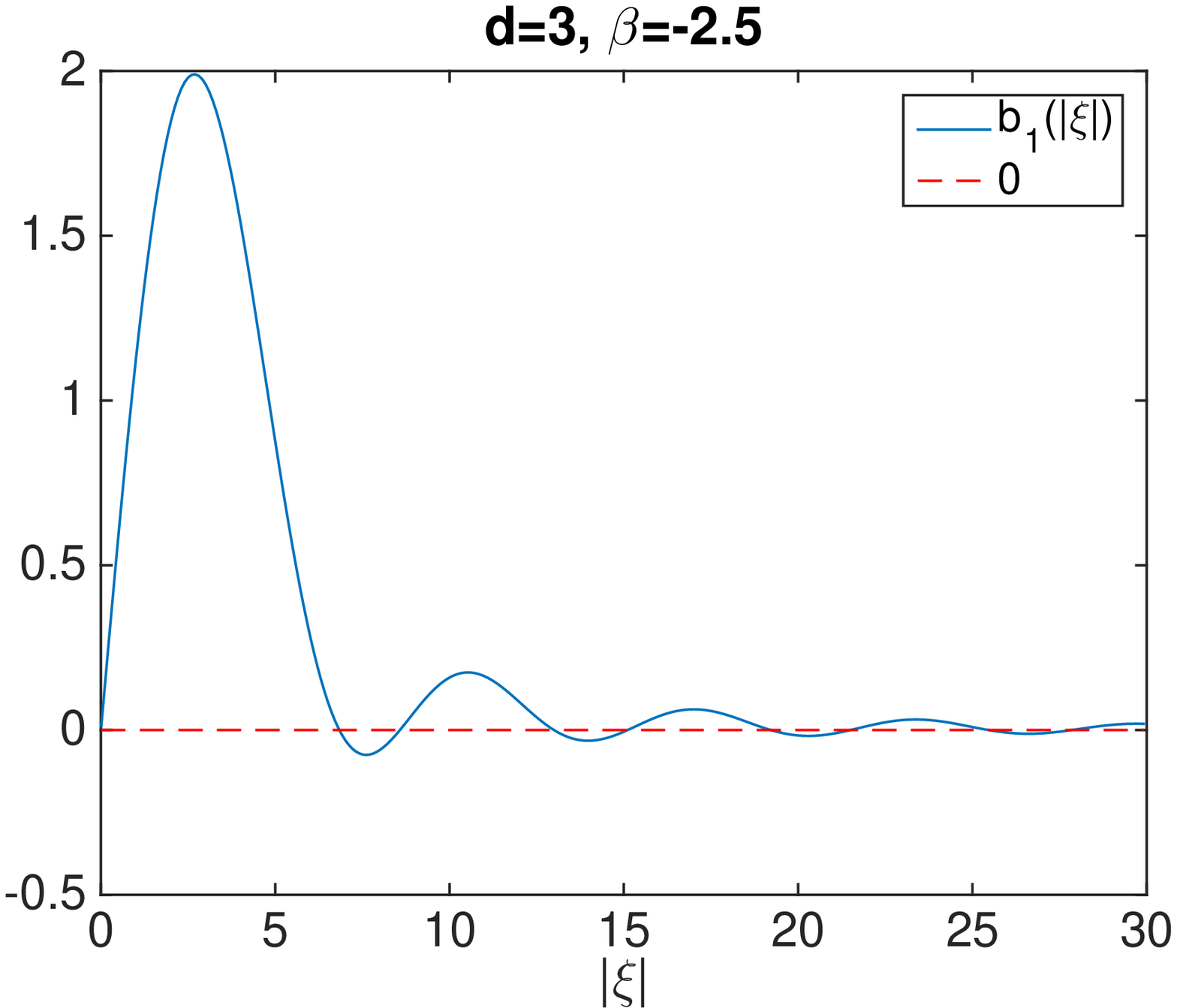}\\
  \includegraphics[width=0.45\textwidth]{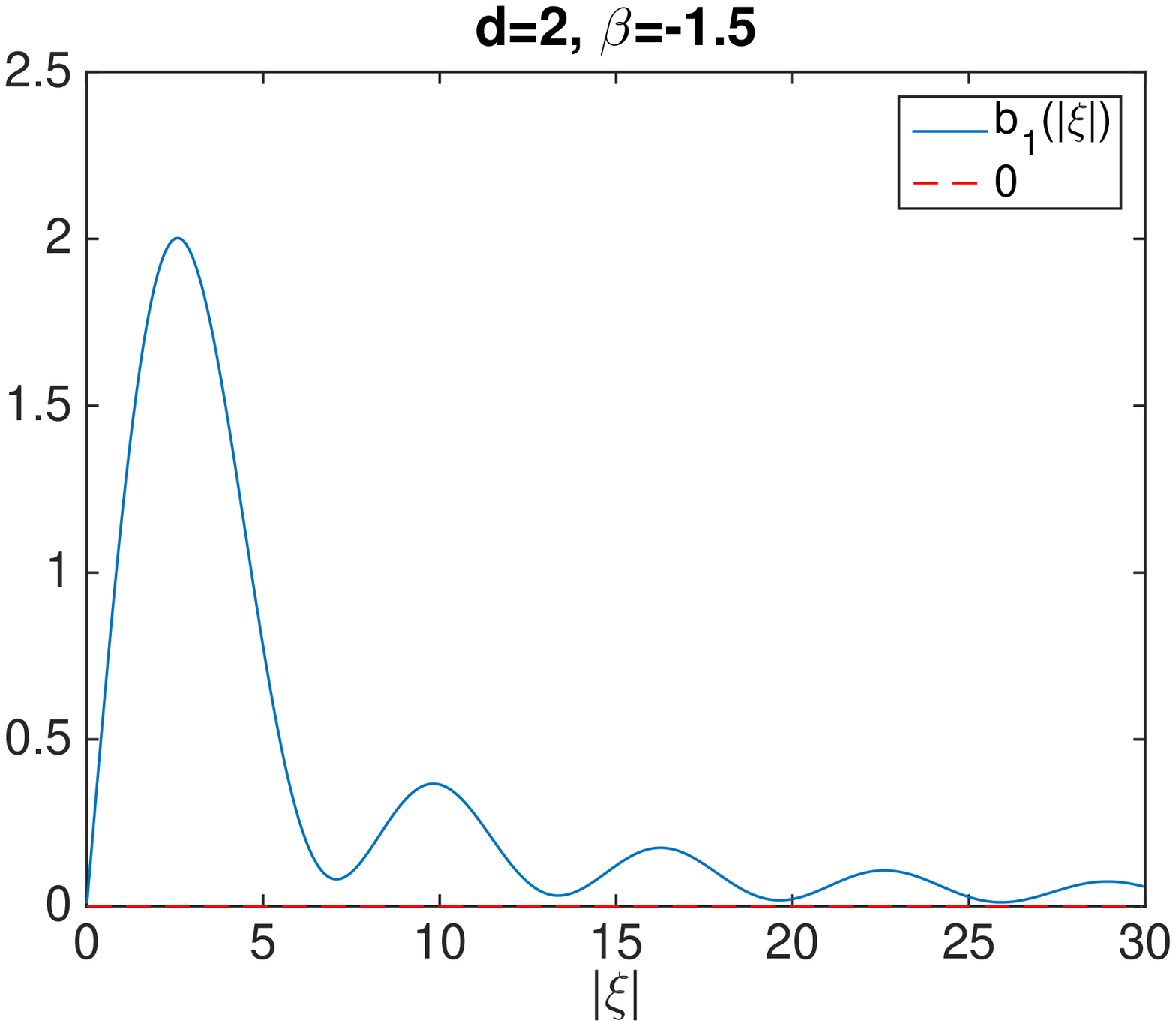}\quad       \includegraphics[width=0.45\textwidth]{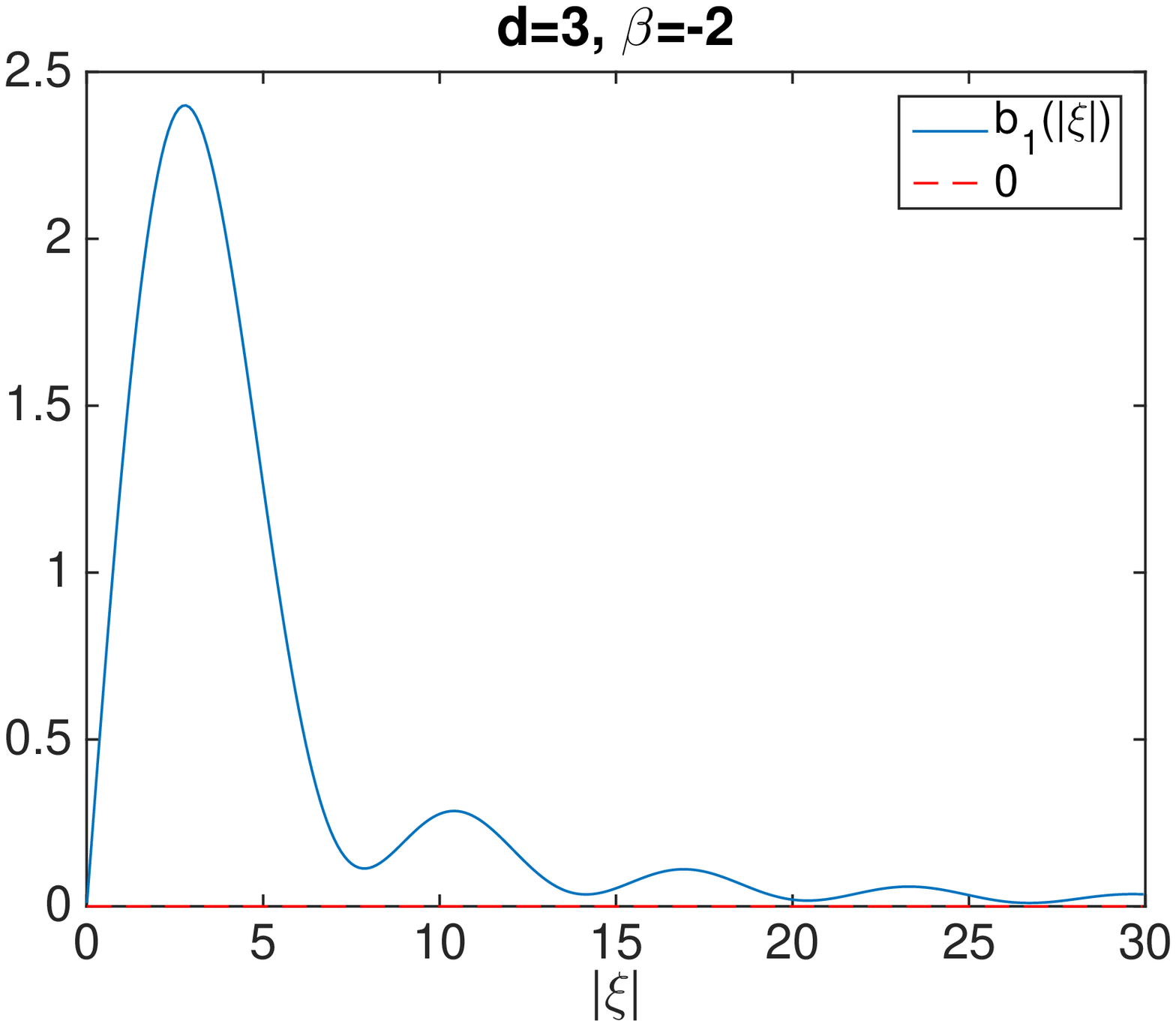}\\
    \includegraphics[width=0.45\textwidth]{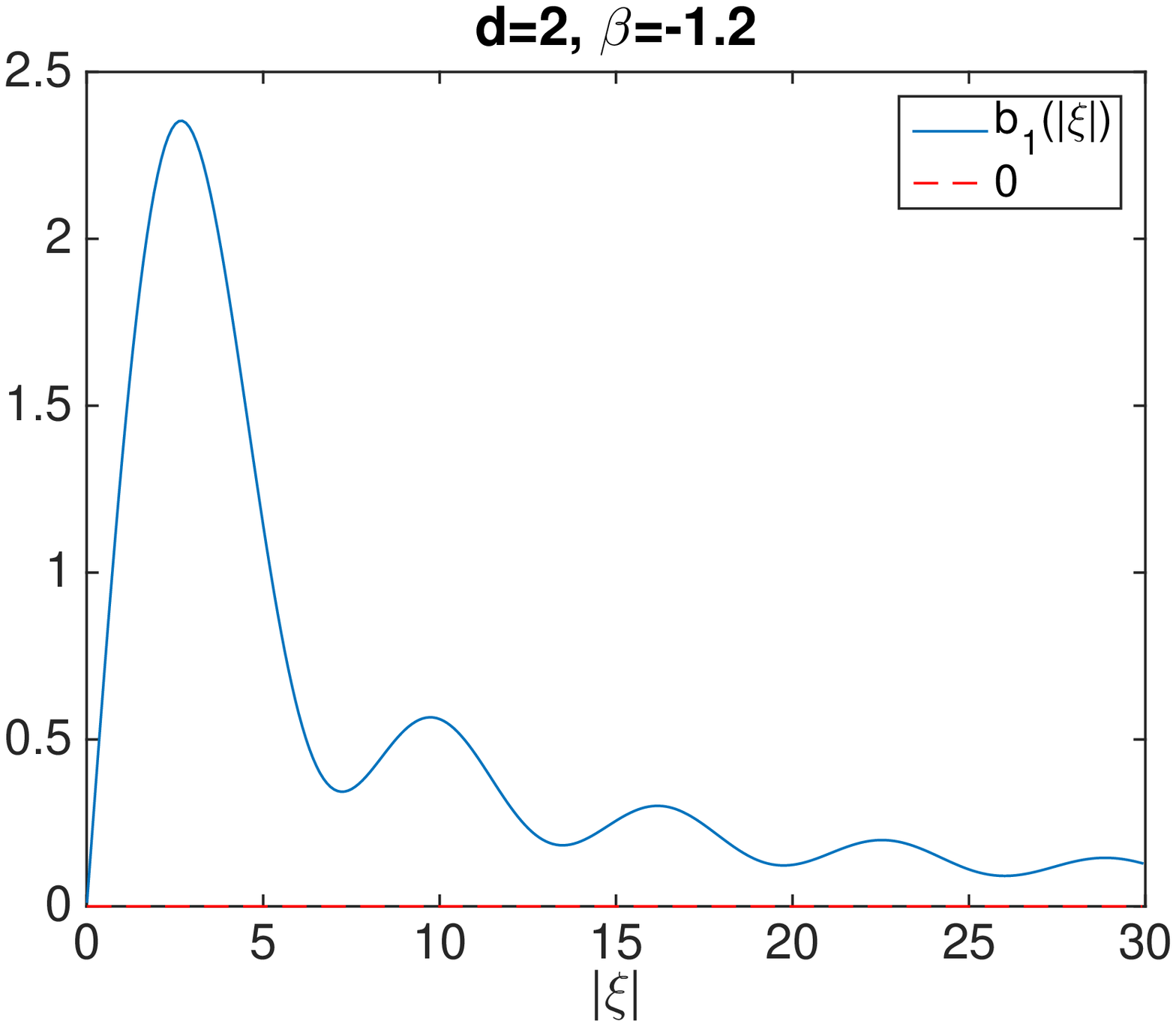}\quad   \includegraphics[width=0.45\textwidth]{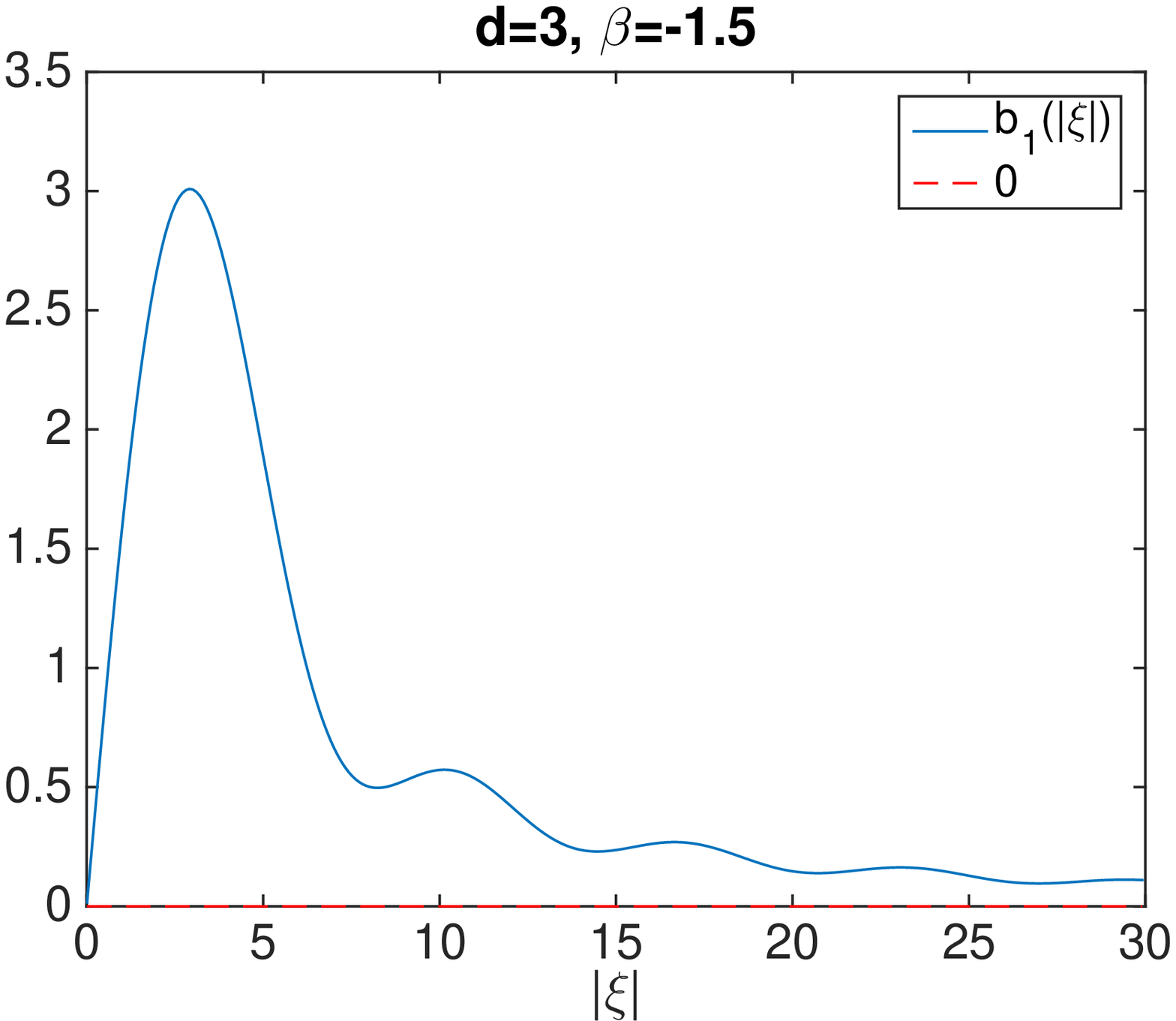}
        \caption{Values of $b_\del(|\bm\xi|)$ against $|\bm\xi|$ for $\del=1$, $d=2$ (left column from top to bottom) with $\beta=-2$ (top), $\beta= -1.5$ (middle), $\beta= -1.2$ (bottom), and $d=3$ (right column from top to bottom) with $\beta=-2.5$ (top),  $\beta= -2$ (middle),  $\beta= -1.5$ (bottom).}
\label{fig:b}
\end{figure}
\end{remark}

The following theorem establishes the existence of a unique solution to the nonlocal Stokes equation.

\begin{thm}
  \label{thm:existence}
Assume that the kernels $\omega_\del$ and $\hat\omega_\del$ satisfy Assumptions~\ref{assu:frackernel} 
and \ref{kernel_rescale}.
Given $\delta>0$, 
there exists a unique solution $(\bfu_\delta, p_\delta)$
to the nonlocal Stokes system \eqref{eq:stokes-nonlocal} with periodic boundary condition 
given in the form of their Fourier series with $(\widehat{\bfu}_\del (\bm\xi), \widehat{p}_\del (\bm\xi))$ computed through
\begin{equation} \label{eq:fouriercoef}
\left( \begin{array}{c} \widehat{\bfu}_\del (\bm\xi) \\  \widehat{p}_\del (\bm\xi) \end{array} \right) = (A_\del(\bm\xi))^{-1} \left( \begin{array}{c} \widehat{\bff}(\bm\xi) \\ 0 \end{array} \right)
\end{equation}
where $(A_\del(\bm\xi))^{-1}$ are defined by \eqref{matrix_inverse}. 
In addition, with $C$ independent of $\del$ and $\bff$ we have
\begin{align}
 \label{estimate_u}  &\|\bfu_\delta\|_{[\cS_\del(\Od)]^d} \leq  C\|\bff\|_{[\cS_\del^\ast(\Od)]^d}   \\
 \label{estimate_p} &\| p_\del\|_{L^2(\Od)} \leq C \| \bff \|_{[H^{-\beta}(\Od)]^d}\,, 
\end{align}
where $\cS_\del(\Od)$ is the energy space with its norm associated with the Fourier symbol $(\lambda_\del(\bm\xi))^{1/2}$ and $\cS_\del^\ast(\Od)$ is its dual space, and 
$\beta\in(-1,1)$ is the exponent defined through \eqref{frackernel}.
\end{thm}

\begin{proof}
From Lemma \ref{lem:positivity} and Assumption \ref{assu:frackernel}, we know that $A_\del(\bm\xi)$ is invertible and the inverse is given by \eqref{matrix_inverse}. 
This gives us
 \[
\widehat{\bfu}_\del (\bm\xi)  = \frac{1}{ \lambda_\del(\bm\xi)} \left( I_d-\frac{\bfb_\del(\bm\xi)\otimes \bfb_\del(\bm\xi)}{|\bfb_\del(\bm\xi)|^2}\right) \widehat{\bff} (\bm\xi) 
\]
and 
\[
\widehat{p}_\del (\bm\xi) = -i \frac{(\bfb_\del(\bm\xi))^T}{|\bfb_\del(\bm\xi)|^2}  \widehat{\bff} (\bm\xi) \,.
\]
So we have
\begin{equation} \label{eq:existthm_1}
|\widehat{\bfu}_\del (\bm\xi)|\leq C \left| \frac{1}{ \lambda_\del(\bm\xi)} \right|  | \widehat{\bff} (\bm\xi) | \,,
\end{equation}
and 
\begin{equation}  \label{eq:existthm_2}
|\widehat{p}_\del (\bm\xi)|\leq  \left| \frac{1}{b_\del(\bm\xi)} \right| | \widehat{\bff} (\bm\xi) |\,.
\end{equation}

From \eqref{eq:existthm_1}, we know immediately that \eqref{estimate_u} is true.  

Now we are left to show \eqref{estimate_p}. From equation \eqref{eq:existthm_2} we only need to estimate $|1/b_\del(|\bm\xi|)|$. We again address the case $d=3$.  
Under the assumption that $r^2\hat\omega(r)$ is nonincreasing, we can use Lemma \ref{lem:positivity} to write
\[
\begin{split}
b_\del(|\bm\xi|)&= 4\pi \int^{\pi/2}_{0}\cos(\phi) \sin(\phi) \int_0^\del r^2 \hod(r) \sin(r \cos(\phi)|\bm\xi| ) dr d\phi  \\
&= \frac{2\pi}{\del} \int^{\pi/2}_{0} \sin(2\phi) \int_0^1 r^2 \hat\omega(r) \sin(r \cos(\phi) \del |\bm\xi|) dr d\phi\,.
\end{split}
\]
Notice that the integral  in the above quantity is positive for any finite $\del |\bm\xi|$ under the assumption that $r^2 \hat\omega(r)$ is nonincreasing.
Indeed, from Lemma \ref{lem:positivity},  we know that the integrand is always nonnegative and it is only possibly zero when
$\cos(\phi) \del |\bm\xi|$ is multiples of $2\pi$, which is a set of measure zero. Thus the integral above is positive for any fixed numbers $\del>0$ and $|\bm\xi|>0$.

Again we denote $a=\del|\bm\xi|$ . 
For $a<1$, we use 
$$\sin(x)\geq x-\frac{x^3}{6}$$ to get
\[
\begin{split}
b_\del(|\bm\xi|) \geq & \frac{2\pi a }{\del} \int_{0}^{\pi/2} \cos(\phi)\sin(2\phi) d\phi \int_0^1 r^3 \hat\omega(r) dr  \\
&\quad -  \frac{2\pi a^3 }{6\del} \int_{0}^{\pi/2} \cos^3(\phi)\sin(2\phi) d\phi \int_0^1 r^5 \hat\omega(r) dr \\
\geq &  \frac{ a }{ \del } C  =C |\bm\xi|\,,
\end{split}
\]
where $C$ is a constant independent of $\delta$.
For $a\in [1, {4\pi}/{\ep}]$ where $\ep$ is the parameter defined in the Assumption \ref{assu:frackernel}, since the integral defined above is positive, it then has a lower bound, namely,  we have
\[
b_\del(|\bm\xi|) \geq  \frac{\tilde C}{\del} \geq \frac{\tilde C \ep}{4\pi} |\bm\xi|\,.
\]
Now for $a> {4\pi }/{\ep}$, we have 
$\cos(\phi) a > {2\pi }/{\ep}$ for $\phi\in(0,\pi/3)$. 
 We then write
\[
\begin{split}
b_\del(|\bm\xi|) & \geq \frac{2\pi}{\del} \int_{0}^{\pi/3} \sin(2\phi) \int_0^1 r^2 \hat\omega(r) \sin(r \cos(\phi) a) dr d\phi  \\
& \geq \frac{2\pi}{\del} \int_{0}^{\pi/3} \sin(2\phi) \left( \int_0^{\frac{2\pi}{\cos(\phi) a}}+\int_{\frac{2\pi}{\cos(\phi) a}}^1 \right) r^2 \hat\omega(r) \sin(r \cos(\phi) a) dr d\phi\,.
\end{split}
\]
Using Lemma \ref{lem:positivity} and the nonincreasing assumption, we observe that
\[
\int_{\frac{2\pi}{\cos(\phi) a}}^1 r^2 \hat\omega(r) \sin(r \cos(\phi) a) dr = \int_0^{1-\frac{2\pi}{\cos(\phi) a}}  h(r)\sin(r \cos(\phi) a) dr \geq 0\,,
\]
where 
\[
h(r): = ( r +\frac{2\pi}{\cos(\phi) a})^2 \hat\omega( r +\frac{2\pi}{\cos(\phi) a})  
\]
is a nonincreasing function.  Then one can show
by using \eqref{frackernel} that
\[ 
\begin{split}
b_\del(|\bm\xi|) &\geq  \frac{2\pi}{\del} \int_{0}^{\pi/3} \sin(2\phi)  \int_0^{\frac{2\pi}{\cos(\phi) a}}r^2 \hat\omega(r) \sin(r \cos(\phi) a) dr d\phi \\
&\geq  \frac{2\pi a^\beta}{\del} \int_{0}^{\pi/3}\cos^{\beta}(\phi) \sin(2\phi) d\phi  \int_0^{2\pi }\frac{1}{r^{1+\beta}}\sin(r ) dr \\
&\geq \frac{C |\bm\xi|^\beta}{\delta^{1-\beta}}
\end{split}
\]
Thus we obtain \eqref{estimate_p}.  %\qed
\end{proof}

An interesting consequence is that, under the specific choices of the
kernel, we see the equivalence 
of a vector field being either locally divergence-free  or
nonlocally divergence-free.

\begin{coro}
  \label{coro:equiv}
Assume that the kernels $\omega_\del$ and $\hat\omega_\del$ satisfy Assumptions~\ref{assu:frackernel}
and \ref{assumption}.
Then, in the distribution sense over the periodic cell,  a periodic and square integrable
 vector field $\bfu$ satisfies $\nabla \cdot \bfu \equiv 0$
if and only if $\mcD_\delta \bfu \equiv 0$. In other  words, we have the  following equivalent function spaces:
$$\{\bfu \in L^2(\Omega)\;:\;  \nabla \cdot \bfu \equiv 0,\;\mbox{in}\;  \Omega.\} 
\equiv  \{\bfu \in L^2(\Omega)\;:\; \mcD_\delta \bfu \equiv 0,\; \mbox{in}\;  \Omega.\} $$
\end{coro}

The above result follows immediately from the 
established positivity of the scalar
 coefficient $b_\delta(\bm\xi)$ for any $\bm\xi\neq \mathbf{0}$.

{
In addition, we may also use the composition of nonlocal divergence and nonlocal gradient to 
get  a nonlocal Laplacian to replace the operator $\mcL_\delta$ in the nonlocal  model \eqref{eq:stokes-nonlocal}.
Similar argument can be adopted to show the well-posedness of the resulting system. We state the conclusion below
without proof.}

\begin{thm}
  \label{thm:existence1}
  {
Assume that the kernel $\hat\omega_\del$ satisfies Assumption \ref{kernel_rescale}.
Given $\delta>0$, 
there exists a unique solution $(\bfu, p)$
to the following modified nonlocal Stokes system 
\begin{align}
  \label{eq:mod-stokes-nonlocal}
\left\{\begin{array}{l}
\D  - \nu \mcD_\delta  \mcG_\delta \bfu_\delta (\bx) + \mcG_\delta p_\delta (\bx) = \bff(\bx), \qquad \bx\in \Om, \\[.2cm]
\D - \mcD_\delta \bfu_\delta(\bx) =0,  \qquad \bx\in \Om, 
\end{array}\right.
\end{align}
with periodic boundary condition and normalization conditions of the type given in \eqref{eq:constraint}.
In addition, with $C$ independent of $\del$ and $\bff$ we have
\begin{align}
 \label{estimate_u_m}  &\|\bfu\|_{[\cV_\del(\Od)]^d} \leq  C\|\bff\|_{[\cV_\del^\ast(\Od)]^d}   \\
 \label{estimate_p_m} &\| p\|_{L^2(\Od)} \leq C \| \bff \|_{[H^{-\beta}(\Od)]^d}\,, 
\end{align}
where $\cV_\del(\Od)$ is the Hilbert space associated with the norm $\|\bfu \|=\{\|\bfu\|^2_{L^2(\Od)} + \| \mcG_\delta \bfu\|^2_{L^2(\Od)}\}^{1/2}$ and  $\cV_\del^\ast(\Od)$ is its dual space, and 
$\beta\in(-1,1)$ is the exponent defined through \eqref{frackernel}.
}
\end{thm}

{
\begin{remark}
Naturally,
we can also establish the well-posedness of the Poisson equation corresponding
to the nonlocal Laplacian $ \mcD_\delta  \mcG_\delta$, just like their local counterparts, under the same conditions given in the above theorem.
In fact, this is the
 usual practice, in the context of solid mechanics, of the correspondence model of peridynamic materials. 
 The study of well-posedness of the latter formulation \cite{dt17} is similar to that carried out there. For the scalar equation,
 the resulting nonlocal interactions  encoded in $\mcL_\delta= \mcD_\delta  \mcG_\delta$
 involve both repulsive and attractive types which is different from the $\mcL_\delta$ operator 
 used in \eqref{eq:op-nonlocal} that features only
 repulsive interactions.
We note that this is also relevant to practical incompressible SPH as
the pressure correction often relies on a well-posed Poisson equation. 
Thus, in case that the kernels  for $\mcD_\delta$ and $\mcG_\delta$ do not have strengthened nearby interactions,  the pressure correction step using $ \mcD_\delta  \mcG_\delta$
might become ill-posed which would also impact the convergence and robustness of the numerical solution. Indeed, it
has been noted that the composition of SPH divergence and SPH gradient leads to a discretization that are sensitive to
particle distributions \cite{cum99,SL03}. From our analysis, we can see that it is no surprise that such phenomenon
does occur as the kernels in the SPH derivatives do not exhibit strong
nearby interactions.
 \end{remark}}

Unlike the case of local elliptic systems, the solutions to nonlocal 
Stokes equation may or may not be more regular than the data ($\bff$ in our case), depending on the specific forms of the nonlocal operators. 
For $\omega=\omega(|\bx|)$ integrable, we can only show that
 the velocity $\bfu_\delta$ remains in $L^2$ if the data $\bff$ 
is also in $L^2$.
On the other hand, in some special cases where
 $\omega=\omega(|\bx|)$ exhibits sufficient singular behavior, 
we can expect some fractional order regularity pick-up. 
For later references, these results are stated below.

\begin{prop}
  \label{thm:pickup}
Assume that the kernels $\omega_\del$ and $\hat\omega_\del$ satisfy 
Assumptions~\ref{assu:frackernel} 
and \ref{kernel_rescale}
with $\omega_\del$ being a
 rescaling of $\omega$. Let $\bfu_\delta$ be the velocity component of the solution to the nonlocal Stokes equation \eqref{eq:stokes-nonlocal}.
Without loss of generality, we also only consider $\delta\in (0,1)$.
If  $\omega(|\bx|)$ is integrable in $\bx$, then we have
\begin{equation} \label{estimate_u1}
 \|\bfu_\delta\|_{[L^2(\Od)]^d}\leq  C\|\bff\|_{[L^2(\Od)]^d} \,.
\end{equation} 
If instead,
\begin{equation} \label{fracdiffusionkernel}
\omega(r)  \geq 
\frac{m}{r^{d+2\al}}
%\leq \frac{M}{r^{d+2\al}}\,, 
\quad \forall r\in (0, 1),
\end{equation} 
for some $\al \in (0,1)$ and a constant $m\in\R^+$, then we have
\begin{equation} \label{estimate_u2}
 \| \bfu_\del\|_{[H^\al(\Od)]^d} \leq C \| \bff\|_{[H^{-\al}(\Od)]^d}\,,
\end{equation}
where $C$ is independent of $\del$ and $\bff$.
\end{prop}

\begin{proof} We follow the proof of the theorem~\ref{thm:existence}.
Without loss of generality, we take the case $d=3$ subject to the additional assumptions of the kernel $\omega(r)$.  From the 
definition of $\lambda_\del(\bm\xi)$, we have
\[
\begin{split}
\lambda_\del(\bm\xi)& \geq 4\pi \int_{0}^{\pi/3} \sin(\phi) \int_0^\del r^2 \omd(r)\big(1- \cos(r \cos(\phi)|\bm\xi|) \big) dr d\phi  \\
&= \frac{4\pi  }{\del^2} \int_{0}^{\pi/3} \sin(\phi) \int_0^1 r^2 \omega(r)\big(1- \cos(r \cos(\phi) \del |\bm\xi|) \big) dr d\phi  
\end{split}
\]
 Now define $a=\del |\bm\xi|$. For $a<1$, we use 
 \[
 \cos(x)\leq 1-\frac{x^2}{2}+\frac{x^4}{24}
 \]
  to get
\[
\begin{split}
\lambda_\del(\bm\xi) \geq  &\frac{4\pi  a^2 }{\del^2} \int_{0}^{\pi/3} \cos^2(\phi)\sin(\phi) d\phi   \int_0^1 r^4 \omega(r) dr \\
& \quad - \frac{4\pi  a^4 }{24 \del^2} \int_{0}^{\pi/3} \cos^4(\phi)\sin(\phi) d\phi   \int_0^1 r^6 \omega(r) dr \\
\geq &  C \frac{ a^2 }{\del^2} =C |\bm\xi|^2\,,
\end{split}
\]
where $C$ is a constant independent of $\del$.
Now we consider the case $a=\del |\bm\xi|\geq1$. Since it is true that for any finite $a$,  $\lambda_\del({\bm\xi})$ is a positive number in the form of $C(a)/\del^2$, where $C(a)$ depends only on $a$,
then $\lambda_\del(\bm\xi)$ has a lower bound $\tilde C/\del^2$ for $a$ belongs to a finite interval with $\tilde C$ being a constant independent of $\del$. 
 
 For the case that $\omega(|\bx|)$ is integrable in $\bx$, we have $r^2\omega(r)$ is integrable in $r$. By using the Riemann Lemma, we can see that  $\lambda_\del (\bm\xi)$ goes to $C/\del^2$ for some constant $C$ as $a\to \infty$.
Thus we have shown \eqref{estimate_u1}. 

As for the case that $\omega(r)$ satisfies \eqref{fracdiffusionkernel}, we have
\[
\begin{split}
\lambda_\del(\bm\xi)&\geq  \frac{4\pi  }{\del^2} \int_{0}^{\pi/3} \sin(\phi) \int_0^1 r^2 \omega(r)\big(1- \cos(r \cos(\phi) \del |\bm\xi|) \big) dr d\phi  \\
&= \frac{4\pi a^{2\al}}{\del^{2}} \int_{0}^{\pi/3} \sin(\phi)  \cos^{2\al}(\phi) \int_0^{\cos(\phi) a} \frac{1}{r^{1+2\al}}\big(1- \cos(r) \big) dr d\phi \\
&\geq \frac{4\pi a^{2\al}}{\del^{2}} \int_{0}^{\pi/3} \sin(\phi)  \cos^{2\al}(\phi) d\phi \int_0^{a/2} \frac{1}{r^{1+2\al}}\big(1- \cos(r) \big) dr\\
& = \frac{C}{\del^{2-2\al}} |\bm\xi|^{2\al}  \,, 
\end{split}
\]
for $a\geq1$. Thus we obtain \eqref{estimate_u2}.
%\qed
\end{proof}

\begin{remark}
With the conditions in  the Assumption~\ref{assu:frackernel}
on the kernel $\hat\omega_\del(r)$, and the more singular behavior imposed in \eqref{frackernel}, we do 
get some regularity pickup on the pressure given in \eqref{estimate_p}.
\end{remark}

Before ending this section, 
we present some additional regularity estimates on the nonlocal solutions for smoother data by observing that 
the nonlocal operators  commute with any local differential operators (and their fractional powers, defined via
the spectrum decomposition)
in the periodic setting. Moreover, we notice that 
the constants in the estimates given in the Theorem~\ref{thm:existence} and Proposition~\ref{thm:pickup}
 are independent of $\del$ and $\bff$, so we can get uniform
regularity estimates stated in the
following corollary.

\begin{coro}
  \label{thm:regularity}
Assume that the kernels $\omega_\del$ and $\hat\omega_\del$ satisfy 
Assumptions~\ref{assu:frackernel} 
and \ref{kernel_rescale}.
The solution $(\bfu_\delta, p_\delta)$
to the nonlocal Stokes system \eqref{eq:stokes-nonlocal} with periodic boundary condition satisfies that
for any partial (and possibly fractional) differential operators $\partial$ on the spatial variables of any nonnegative order,
\begin{align}
 \label{estimate_ur}  &\|\partial \bfu_\delta\|_{[\cS_\del(\Od)]^d} \leq  C\| \partial \bff\|_{[\cS_\del^\ast(\Od)]^d} \,,  \\
 \label{estimate_pr} &\|\partial p_\del\|_{L^2(\Od)} \leq C \| \partial \bff \|_{[H^{-\beta}(\Od)]^d}\,, 
\end{align}
where $C>0$ is a generic constant independent of $\del$, $\bff$ and $\partial$. Moreover, if $\omega(r)$ 
satisfies \eqref{fracdiffusionkernel} for some $\al \in (0,1)$ and constants $m, M \in\R^+$,  then
\begin{align}
 \label{estimate_urf}  &\|\partial \bfu_\delta\|_{[H^\al(\Od)]^d} \leq  C\| \partial \bff\|_{[H^{-\al}(\Od)]^d}\,,   \\
 \label{estimate_prf} &\|\partial p_\del\|_{L^2(\Od)} \leq C \| \partial \bff \|_{[H^{-\beta}(\Od)]^d}\,, 
\end{align}
for a generic constant $C>0$ that is  independent of $\del$, $\bff$ and $\partial$. 
\end{coro}

%%%%%%%%%%%%%%%%%%%%%%%%%%%%%%%%%%%%%%%%%%%%%%%%%%%%%%%%%%%%%%%
%%%%%%%%%%%%%%%%%%%%%%%%%%%%%%%%%%%%%%%%%%%%%%%%%%%%%%%%%%%%%%%
%%%%%%%%%%%%%%%%%%%%%%%%%%%%%%%%%%%%%%%%%%%%%%%%%%%%%%%%%%%%%%%

\section{{The local limit}}
\label{sec:vanish}
 
Since the nonlocal operators, as defined here,  are constructed to have the corresponding differential operators as the local limits
as the horizon
(smoothing length) $\del$ shrinks to zero, it is reasonable to expect that the limit of the 
nonlocal Stokes system \eqref{eq:stokes-nonlocal} recovers the conventional local Stokes system
as nonlocal effects vanish.  With the energy  estimates shown earlier, it is possible to derive rigorously the zero $\del$ limit of 
 \eqref{eq:stokes-nonlocal}. Moreover, we can
establish the convergence rate of the nonlocal solutions to its local counterpart as $\del\to0$ using Fourier analysis.

\begin{thm}
  \label{thm:converge}
  Assume that the kernels $\omega_\del$ and $\hat\omega_\del$ satisfy 
Assumptions~\ref{assu:frackernel} 
and \ref{kernel_rescale}.
Let $(\bfu, p$) be the solution of Stokes system \eqref{eq:stokes} and $(\bfu_\delta, p_\delta)$ be the solution of nonlocal Stokes system \eqref{eq:stokes-nonlocal}.
Under the periodic conditions and the Assumptions~\ref{assu:frackernel} 
and \ref{kernel_rescale}
 on the kernels, there is a constant
 $C$ independent of $\del $ and $\bff$ such that 
\begin{align}
  \|\bfu-\bfu_\delta\|_{[L^2(\Od)]^d}&\leq C\del^2 \|\bff\|_{[L^2(\Od)]^d}\label{error_u}\,.
\end{align}
and
\begin{align} \label{error_p}
 \|p-p_\delta\|_{L^2(\Od)}&\leq  C\delta^{\min\{2, 1+\eta\}}  \|\bff\|_{[H^{\eta}(\Od)]^d} \quad \text{for any } \eta\geq -\beta \,,
 \end{align}
 where $\beta\in(-1,1)$ is the exponent defined through \eqref{frackernel}.
  \end{thm}
 
\begin{proof} Let us work on the case $d=3$ as illustration.
We may again obtain from \eqref{matrix_inverse} the estimate of the Fourier coefficients:
\[
\begin{split}
 |\widehat{\bfu}(\bm\xi)-\widehat{\bfu}_\delta(\bm\xi) | &\leq  \left( \left|\frac{1}{ \lambda_\del(\bm\xi)}-\frac{1}{ |\bm\xi|^2}\right|+
  \left|\frac{\bfb_\del(\bm\xi)\otimes \bfb_\del(\bm\xi)}{\lambda_\del(\bm\xi) |\bfb_\del(\bm\xi)|^2} -\frac{ \bm\xi\otimes \bm\xi}{|\bm\xi|^4} \right| \right) | \widehat{\bff} (\bm\xi) | \\
  &\leq C    \left|\frac{1}{ \lambda_\del(\bm\xi)}-\frac{1}{ |\bm\xi|^2}\right| | \widehat{\bff} (\bm\xi) |
  \end{split}
\]
and
\[
|\widehat{p}(\bm\xi)-\widehat{p}_\delta(\bm\xi)|\leq C  \left|\frac{\bfb_\del(\bm\xi)}{|\bfb_\del(\bm\xi)|}-\frac{\bm\xi}{ |\bm\xi|^2}\right| | \widehat{\bff} (\bm\xi) | \,.
\]

Then \eqref{error_u} is  just a consequence of the following estimate of the difference between $1/\lambda_\del(\bm\xi)$ and $1/|\bm\xi|^2$,
which we can draw similar arguments from \cite[Lemma 1]{duyang16} to obtain:
\[
\left|\frac{1}{ \lambda_\del(\bm\xi)}-\frac{1}{ |\bm\xi|^2}\right| \leq C \del^2\,,
\]
where $C$ is a constant independent of $\del$ and $\bm \xi$.

To get the proof of \eqref{error_p},  we notice that  for $|\bm\xi|\geq1$,
\[
 \left|\frac{\bfb_\del(\bm\xi)}{|\bfb_\del(\bm\xi)|^2}-\frac{\bm\xi}{ |\bm\xi|^2}\right|   =
  \left|\frac{\bm\xi}{|\bm\xi|} \left( \frac{1}{b_\del(|\bm\xi|)}-\frac{1}{ |\bm\xi|}\right)\right| \leq  \left| \frac{1}{b_\del(|\bm\xi|)}-\frac{1}{ |\bm\xi|}\right|\,, 
\]
where $b_\del(|\bm\xi|)$ is given by \eqref{coef_b_scalar}.
Let $a=|\bm\xi|\del$, then
\[
\left| \frac{1}{ b_\del(|\bm\xi|)}-\frac{1}{ |\bm\xi|}\right|=
\del \left| \D \frac{1}{2\pi \int_0^{\pi/2}  \int_0^1 r^2 \hat\omega(r) \sin(2\phi) \sin(r \cos(\phi)a)drd\phi}-\frac{1}{ a}\right|\,.
\]
For $a<1$, since
\[
r \cos(\phi)a-\frac{(r \cos(\phi)a)^3}{3!}\leq \sin(r \cos(\phi)a)\leq r \cos(\phi)a\,,
\]
we obtain
\[
a-\frac{a^3}{3!}\leq 2\pi \int_0^{\pi/2} \sin(2\phi) \int_0^1 r^2 \hat\omega(r)\sin(r \cos(\phi)a)drd\phi \leq  a\,.
\]
So
\[
\frac{1}{\delta} \left| \frac{1}{ b_\del(|\bm\xi|)}-\frac{1}{ |\bm\xi|}\right|
%=\left| \frac{1}{2\pi \int_0^{\pi/2} \sin(2\phi) \int_0^1 r^2 \hat\omega(r)\sin(r \cos(\phi)a)drd\phi}-\frac{1}{ a}\right|
\leq  \frac{1}{a-\frac{a^3}{3!}}-\frac{1}{a}=\frac{a}{6-a^2}\leq a\,.
\]
So we have
\begin{equation}
\left| \frac{1}{b_\del(|\bm\xi|)}-\frac{1}{ |\bm\xi|}\right|\leq \delta a = \del^2 |\bm\xi|\,,
\end{equation}
which implies that 
\begin{equation} 
\left| \frac{1}{b_\del(|\bm\xi|)}-\frac{1}{ |\bm\xi|}\right| \frac{1}{|\bm\xi|^\eta}\leq C \del^{\min\{2, 1+\eta\}} |\bm\xi|^{\min\{0, 1-\eta\}} \leq C \del^{\min\{2, 1+\eta\}} \,,
\end{equation}
for the case $a=\del|\bm\xi|<1$.
For $a\geq1$, we proceed the same way as in the proof of Theorem \ref{thm:existence}  to obtain 
\[
b_\del(|\bm\xi|)
\geq\left\{
\begin{aligned}
&C |\bm\xi| \quad \text{ for } \del |\bm\xi | \in [1,4\pi/\ep] \\
& C \frac{|\bm\xi|^\beta}{\del^{1-\beta}} \quad \text{ for } \del |\bm\xi | \in (4\pi/\ep, \infty) \,.
\end{aligned}
\right.
\]
Then we have for $a=\del|\bm\xi|\in [1,4\pi/\ep]$,
\[
\left| \frac{1}{b_\del(|\bm\xi|)}-\frac{1}{ |\bm\xi|}\right| \frac{1}{|\bm\xi|^\eta}\leq C\frac{1}{|\bm\xi|^{1+\eta}}\leq C \del^{1+\eta}\,.
\]
And for the case  $a=\del|\bm\xi|\geq 4\pi/\ep$, we use the assumption that $\eta\geq - \beta$ to obtain
\[
\left| \frac{1}{b_\del(|\bm\xi|)}-\frac{1}{ |\bm\xi|}\right| \frac{1}{|\bm\xi|^\eta}\leq  C \frac{\del^{1-\beta}}{|\bm\xi|^{\beta+\eta}}+\frac{1}{|\bm\xi|^{1+\eta}} \leq C \del^{1-\beta+\beta+\eta}+\del^{1+\eta} \leq \tilde C \del^{1+\eta}\,.
\]
Combine the above arguments we arrive at \eqref{error_p}.  %\qed
 \end{proof}

\section{Numerical discretization} 
\label{sec:num}
With a well-posed nonlocal Stokes system \eqref{eq:stokes-nonlocal}, one may
readily consider its numerical discretization. We leave the discussion on its
particle approximation and the connection to the incompressible SPH to future works
due to the need for more lengthy derivations.
Instead, under periodic conditions, 
it is natural to consider Fourier spectral method for numerical approximation whose convergence can be
subject to the similar Fourier analysis.

Let $(\bfu_\delta^N, p_\delta^N)$ stands for the Fourier spectral approximation of $(\bfu_\delta, p_\delta)$.
It is easy to see that $(\bfu_\delta^N, p_\delta^N)$ are simply the truncation (projection) of $(\bfu_\delta, p_\delta)$
over all Fourier modes with wave numbers no larger than $N$. 
Hence, we get the following convergence result for a fixed $\delta>0$ as $N\to \infty$ for $\bff\in [L^2(\Od)]^3$
and also convergence rates for smoother data.

\begin{thm} \label{thm:fourierconverge-nl}
Let $(\bfu_\delta^N(\bx),\; p_\delta^N(\bx))$  be Fourier spectral approximation to the \eqref{eq:stokes-nonlocal}. 
We assume also that Assumptions~\ref{assu:frackernel} and \ref{kernel_rescale} hold true for the kernels. Then 
for $\bff\in [L^2(\Od)]^3\cap [H^{-\beta}(\Od)]^3$,  we have as $N\to \infty$, 
\begin{align}
  \|\bfu_\delta^N-\bfu_\delta \|_{[L^2(\Od)]^3}\to 0
\quad \text{ and } \quad
 \|p_\delta^N-p_\delta \|_{L^2(\Od)}\to 0\,.
 \end{align}
 Moreover,   
 if $\omega(r)$ 
satisfies \eqref{fracdiffusionkernel} for some $\al \in (0,1)$ and constants $m, M \in\R^+$,  then for any $s\geq 0$, 
with $C$ independent of $\del $, $N$, $\bff$ and $s$,  we have as $N\to \infty$, 
\begin{align}
  \|\bfu_\delta^N-\bfu_\delta \|_{[H^{\gamma+\al}(\Od)]^d} &\leq \frac{C}{N^s} \|\partial \bff \|_{[H^{\gamma-\al}(\Od)]^d}
\end{align}
and
\begin{align}
 \|p_\delta^N-p_\delta \|_{H^{\gamma}(\Od)}&\leq   \frac{C}{N^s}  \|\partial  \bff \|_{[H^{\gamma-\beta}(\Od)]^d}\,,
 \end{align}
  where $s>0$ denote the total order of differentiations of a partial differential operator $\partial$.
\end{thm}

\begin{proof}
The results follow from standard Fourier analysis and the regularity estimates in Corollary \ref{thm:regularity}. %\qed
\end{proof}

\begin{remark}
By Corollary \ref{coro:equiv}, we can also see that the discrete Fourier spectral approximation automatically leads to a divergence-free vector fields under the Assumptions~\ref{assu:frackernel} and \ref{kernel_rescale} on the kernels, simultaneously in the local and nonlocal sense.
\end{remark}

Although nonlocal models may be of interests in their own right, given that they have been used as integral relaxations to the local models, we would like to study the convergence properties of nonlocal discrete solutions to solutions of
the corresponding  local continuum models. Along this direction, 
we would like to emphasize the fact that the Fourier spectral approximation for \eqref{eq:stokes-nonlocal} is 
asymptotically compatible (a notion developed in \cite{TiDu14}), in the sense that it is not only a convergent numerical method for the nonlocal
problem with fixed $\del$, but also preserves the  compatibility to the asymptotic local limit as $\del$ shrinks to zero.
Having the asymptotic compatibility provides robustness to the numerical discretization since the numerical
solution in various parameter regimes (that involving both the smoothing length and the spatial discretization parameter)
are expected to converge to the desired continuum limit with increased numerical resolution.
 In mathematical terms, one 
expects the following to be true:
\beq
\left\{
\begin{array}{c}
\| \bfu_\delta^N - \bfu \| \to 0,\\
\text{ and }\\
\|p_\delta^N-p \|\to0,
\end{array}
\right.\quad
 \text{ as }\quad
 \left\{
\begin{array}{c}
 \del\to0,\\
 \text{ and }\\
  N\to\infty\,.
 \end{array}
\right. 
\eeq
By the projection property. we have
\beq \label{proju}
\| \bfu_\delta^N - \bfu^N \| \leq
\| \bfu_\delta - \bfu \|
\eeq
and
\beq \label{projp}
 \| p_\delta^N - p^N \| \leq
  \| p_\delta - p \|\,.
  \eeq
Thus, one way to derive the convergence to the local limit is through the triangle inequalities:
\beq \label{triangle_ineq}
\begin{aligned}
\| \bfu_\delta^N - \bfu \| & \leq \| \bfu_\delta^N - \bfu^N \|+ \| \bfu^N - \bfu \| \\
& \leq \| \bfu_\delta - \bfu \|+ \| \bfu^N - \bfu \| 
\end{aligned}
\eeq
and
\beq \label{triangle_ineq1}
\begin{aligned}
 \| p_\delta^N - p \| & \leq \| p_\delta^N - p^N \|+ \| p^N - p \|\\
  & \leq \| p_\delta - p \|+ \| p^N - p \|
\end{aligned}
\eeq
where $(\bfu^N, p^N)$ denotes the Fourier approximation for the standard Stokes equation, which converges to $(\bfu, p)$ as $N\to\infty$. 

Now to visualize the asymptotic compatibility of the Fourier spectral approximation, we present
a diagram in Fig \ref{fig:diagram} showing the different paths of convergence, following the work of \cite{TiDu14}.

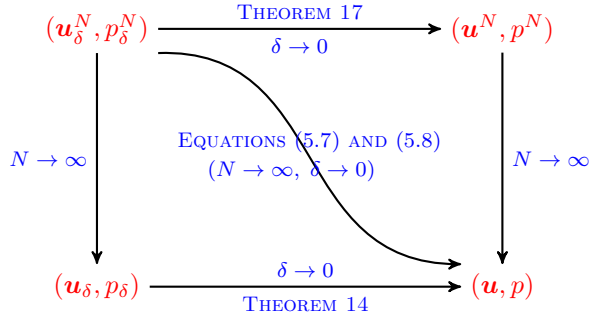
\begin{figure}[htbp]
  \centering
       \begin{tikzpicture}[scale=0.98]
     \tikzset{to/.style={->,>=stealth',line width=.8pt}}   
 \node(v1) at (0,3.5) {\textcolor{red}{$(\bfu_\delta^N, p_\delta^N)$}};
  \node (v2) at (5.5,3.5) {\textcolor{red}{$(\bfu^N, p^N)$}};
   \node (v3) at (0,0) {\textcolor{red}{$(\bfu_\delta, p_\delta)$}};
    \node (v4) at (5.5,0) {\textcolor{red}{$(\bfu, p)$}};
     \draw[to] (v1.east) -- node[midway,above] {\footnotesize{\textcolor{blue}{\sc  Theorem \ref{thm:fourierconverge}}}}  node[midway,below] {\footnotesize{\textcolor{blue}{$\delta\to0$}}}      
     (v2.west);
     \draw[to] (v1.south) -- node[midway,left] {\footnotesize{\textcolor{blue}{$N\to\infty$}}} (v3.north);
       \draw[to] (v3.east) -- node[midway,below] {\footnotesize{\textcolor{blue}{\sc Theorem \ref{thm:converge}}}} node[midway,above] {\footnotesize{\textcolor{blue}{$\del\to0$}}} (v4.west);
       \draw[to] (v2.south) -- node[midway,right] {\footnotesize{\textcolor{blue}{$N\to\infty$}}}(v4.north);
     \draw[to] (v1.south east) to[out = 2, in = 180, looseness = 1.2] node[midway]  {\footnotesize\hbox{\shortstack[l]{ {\textcolor{blue}{\sc  Equations \eqref{triangle_ineq} and \eqref{triangle_ineq1}}}\\ { \quad \textcolor{blue}{($N\to\infty,\;\del\to0$)}} }}} (v4.north west);

      \end{tikzpicture}
   \caption{A diagram for asymptotic compatibility and convergence results.}\label{fig:diagram}
\end{figure}

\begin{thm} \label{thm:fourierconverge}
Suppose $(\bfu_\delta^N(\bx),\; p_\delta^N(\bx))$ and $(\bfu^N(\bx),\; p^N(\bx))$ are Fourier spectral approximations to the \eqref{eq:stokes-nonlocal} and \eqref{eq:stokes} respectively. 
We assume also that Assumptions~\ref{assu:frackernel} 
and \ref{kernel_rescale}
 hold true for the kernels. Then with $C$ independent 
of $\del $, $N$ and $\bff$, we have
\begin{align}
  \|\bfu_\delta^N-\bfu^N\|_{[L^2(\Od)]^3}&\leq C\del^2 \|\bff^N\|_{[L^2(\Od)]^3}\,,
\end{align}
and
\begin{align}
 \|p_\delta^N-p^N\|_{L^2(\Od)}&\leq  C\delta^{\min\{2, 1+\eta\}}  \|\bff^N\|_{[H^{\eta}(\Od)]^3} \quad \text{for }\eta\geq -\beta \,,
 \end{align}
 where $\beta\in(-1,1)$ is the exponent defined through \eqref{frackernel}.
\end{thm}
\begin{proof}
The result is immediate from \eqref{proju} and \eqref{projp} and Theorem \ref{thm:converge}. %\qed
\end{proof}

While the focus of this work is mainly on theoretical analysis, the asymptotic compatibility given in the Theorem
\ref{thm:fourierconverge} reveals interesting possibilities to design numerical discretization of the local Stokes
equation via nonlocal integral relaxations without imposing the smoothing length $\delta$ to be proportional to  $h=1/N$ with
$N$ representing the discretization parameter (or $h$ representing the scale of numerical resolution).
The latter, with $h$ being the parameter for typical particle spacing, is a common practice in methods like SPH. 
Relaxing such constraints can potentially lead to
more effective and robust approximations especially when simulating complex flow patterns that require
more adaptive choices of smoothing and discretization. Finding convergent approximation for particle like
approximation to the integral formulations for more general $h$ and $\delta$ has been systematically
studied and successfully explored computationally 
\cite{ELC02,SRS10}, though the focus there were  corrections on the discrete level to assure reproducing conditions, which
has been a popular approach developed in meshfree approximation literature.
Further theoretical investigations on the connections of these related ideas will be carried out in future works.

\section{Conclusion and discussions}
\label{sec:discu}
Recent development of nonlocal continuum models and nonlocal calculus has provided us 
useful tools to better understand models and numerical methods that may involve nonlocality, either
on the physical level or for convenience of numerical computation. A number of studies have been
carried out for solid mechanics in the context of peridynamics.  {This work 
is an attempt to extend the mathematical study of nonlocal models to fluid mechanics. It is mainly aimed at 
providing new theoretical insight to popular methods for simulating fluid flows such as SPH and vortex-blob methods.
The latter two approaches have the notion of nonlocality and integral relaxations explicitly built in their formulations. }
In addition to being
a computational technique, the introduction of nonlocality can also arise due to physical considerations
such as the nonlocal memory effect in viscoelastic fluids and the nonlocal spatial effect in quasi-geostrophic flows.
{Indeed, the message we want to deliver in this paper is that one should perhaps first study the continuum relaxation of the 
PDEs more systematically before designing consistent, stable and robust numerical discretization.  
It is thus a meaningful mathematical exercise to consider a well-posed nonlocal analog of the
local continuum equations as one foundation  block for understanding  and improving the relevant numerical methods.}
The setting presented here is simplified to the case of nonlocal Stokes equation with periodic boundary conditions so that we can first probe the choices of  appropriate nonlocal operators without delving into other technical difficulties.
When replacing the local differential operators by their nonlocal counterparts such as those introduced in previous works \cite{du13,MeDu15}, 
we see that one should be particularly careful about the interaction kernels used in  nonlocal gradient and divergence operators, in the sense that 
the interaction should be sufficiently strengthened to make the system stay solvable and stable. 
The condition imposed, as elucidated in the work, is a natural one that makes the spaces of local and nonlocal divergence free
vector fields equivalent. The resulting models with proper nonlocal
interaction kernels are not only well-posed for any finite horizon and smoothing length, but their solutions also converge to the solution of
classical Stokes  equation in the local limit, a fact established in this work along with precise convergence rates. 

The nonlocal Stokes equation, as a continuum model, serves as a bridge between the local
Stokes equation and its discretizations like SPH. In building such a linkage, 
the notion of asymptotically compatible schemes shown in Fig \ref{fig:diagram} can become important for practical applications
due to the implied robustness of the underlying numerical methods so 
that one does not necessarily need  the discretization to be refined faster than the reduction of smoothing length. 
In particular, the Fourier spectral method is shown to enjoy asymptotic compatibility.  It is certainly more interesting to look into other numerical methods, 
in particular, particle discretizations like SPH, which is a main objective of our ongoing series of 
works. Although no extended investigations are made here on either mesh 
or particle based discretization, some preliminary speculations can be offered. {For example, while we leave more
detailed studies to subsequent works, 
it is no surprising that the well-posedness of the continuum nonlocal Stokes equation does not guarantee that simple minded discretization is automatically stable. 
As a comparison, it is well-known that to solve the conventional local Stokes equation based on standard centered finite differences on a Cartesian mesh, 
check-board type instabilities may arise when the unknown velocity and pressure are placed at the same set of mesh points. Instead, 
 the so-called MAC (Marker and Cell) type  schemes are developed to
 place the unknown variables at staggered locations. Such issues can be studied in the nonlocal setting, very much in the same spirit
 of the work presented here. }
 For an brief illustration, consider the one dimensional 
 % One expects to have similar stories for the nonlocal Stokes equation. 
% The point can be
% briefly illustrated in simple terms in one space dimension. Indeed, suppose we have a single mesh for $\{ x_j\} $ on $(-\pi, \pi)$ for $u$ and $p$, and the 
nonlocal gradient operator  given by
\[
\mcG_\delta p(x)= \int_0^\del  \hat\omega_\del(s) (p(x+s) - p(x-s)) ds\,,
\]
we have naturally two types of finite difference approximation, one with a regular uniform grid and another with a staggered grid that are given in the following respectively: 
\begin{align*}
&\text{Regular: } \quad \mcG^h_\delta p(x_j)=\sum_{k=0}^r  d_k \left( p(x_{j+k})- p(x_{j-k})\right)\\
&\text{Staggered: } \quad \mcG^h_\delta p(x_j)=\sum_{k=0}^r  d_k \left( p(x_{j+k+\frac{1}{2}})- p(x_{j-k-\frac{1}{2}})\right)
\end{align*}
for some nonnegative coefficients $\{d_k\}$. Again by Fourier analysis, we may find that the eigenvalues of the two discrete operators are:
\[
i b_{\del,h} (n) = \left\{
\begin{aligned} 
&i 2 \sum_{k=0}^r d_k \sin(n k h) \quad \text{for regular grid} \\ 
& i 2 \sum_{k=0}^r d_k \sin(nkh+{nh}/{2})  \quad \text{for staggered grid}
\end{aligned}
\right.
\]
where $h={2\pi}/{N}$ and $n=-N+1, \cdots, N$.   One can observe that 
for the regular grid the discrete eigenvalue is zero if $n=N$ while it is generically not the case for the staggered case.
The story of numerical stability is then different for the two discrete operators, see \cite{dlee19} for additional discussions on
multidimensional cases. 
%So we find that even in the continuum setting we impose singularity on the kernel $\hat\omega_\del$ to make sure that the eigenvalues of $\mcG_\delta$ stay nonzero, 
%we still encounter with zero discrete eigenvalue by choosing $n=N$ in the above expression. Thus the design of a stable finite difference scheme is again important for solving the nonlocal Stokes equation. 
%We proposed and studied the nonlocal Stokes equation, but there many unsolved questions that can be studied in the future along this line.  

Finally, there are various possible extensions of the work here. Our nonlocal formulation here is based on a centered nonlocal relaxation to
local differential operators. For example, for $\mcG_\delta$ in 
\eqref{eq:op-nonlocal} , it is determined by a vector $\vod(\by-\bx)$ which can be viewed as an odd and rank-one tensor. A more general form of such
a nonlocal gradient operator acting on a vector field $\bfu$, by the Schwartz kernel theorem, can be written as a second order tensor \cite{MeDu15} given by
\[
\mcG_\del{\bfu}(\bx   ):=  \int \rho_\delta(|\by -\bx   |)\,\mathsf{M}_\delta(\by -\bx   ) \, \frac{{\bfu}(\by )-{\bfu}(\bx   )}{|\by -\bx   |} \;d\by ,\;\;
\]
where  $\mathsf{M}_\delta$  is a 3rd-order odd tensor and the integral is interpreted in the principal sense.
 There are also other forms of nonlocal operators such as those based on non-symmetric or one-sided interaction kernels.
 For example, the one dimensional forward/backward nonlocal differentiation operators
  \[
\mcG_\delta^+ p(x)= \int_0^\del  \rho_\del(s) (p(x+s) - p(x)) ds\,,
\quad
\mcG_\delta^- p(x)= \int_0^\del  \rho_\del(s) (p(x) - p(x-s)) ds\,,
\]
 have been studied in \cite{dyz16} that  can
 lead to invertible operators for a wider class of nonlocal interaction kernels $\rho_\del=\rho_\del(s)$ with orientation 
 bias \cite{dlee19}.  Moreover, one may consider formulations involving stabilization terms by introducing some nonlocal analog of artificial compressibility similar
to  ideas used for conventional local Stokes models \cite{zds19}.
Studying the nonlocal analog of time-dependent nonlinear Navier-Stokes system is surely
another step to take \cite{dlee19}. Extensions can also be considered for compressible flows, interfacial and multiphase flows, magneto-hydrodynamics
(MHD) and stochastically
driven flows. The extension in the case of a scalar hyperbolic conservation laws can be found in \cite{ld19}.
For SPH, much effort has also been devoted to the accurate treatment of boundary conditions, even though particle like meshfree methods are thought
to be able to handle complex geometry and boundary conditions more effectively. This is perhaps not surprising, given the
intrinsic nonlocal continuum formulations explicitly formulated here that demand generically the notion of nonlocal boundary conditions
or volumetric constraints (see \cite{du12,MeDu15}). {It leads to
another important  topic of future research.  Naturally, once the model and discretization are in place, there are still many practical  issues ranging from quadratures to linear and nonlinear solvers that must also be investigated along with more careful theoretical analysis.}

\vspace{0.2cm}

%%%%%%%%%%%%%%%%%%%%%COMMENT%%%%%%%%%%%%%%%%%%%%%%%%%%%%%%%

\noindent{\bf Acknowledgement}: The authors would like to thank  R. Lehoucq, J. Foster and  A.  Tartakovsky for discussions on related subjects.

\end{document}